\documentclass[12pt, twoside]{article}
\usepackage{amsmath,amsthm,amssymb,color}
\usepackage{times}
\usepackage{enumerate}

\usepackage{mathrsfs}
\usepackage{graphics}

\def\titlerunning#1{\gdef\titrun{#1}}
\makeatletter
\def\author#1{\gdef\autrun{\def\and{\unskip, }#1}\gdef\@author{#1}}
\def\address#1{{\def\and{\\\hspace*{18pt}}\renewcommand{\thefootnote}{}%
\footnote {#1}}%
\markboth{\autrun}{\titrun}}
\makeatother
\def\email#1{e-mail: #1}
\def\subjclass#1{{\renewcommand{\thefootnote}{}%
\footnote{\emph{Mathematics Subject Classification (2010):} #1}}}

\newtheorem{thm}{Theorem}[section]
\newtheorem{cor}[thm]{Corollary}
\newtheorem{prop}[thm]{Proposition}

\theoremstyle{definition}
\newtheorem{defin}[thm]{Definition}
\newtheorem{rem}[thm]{Remark}
\newtheorem{exa}[thm]{Example}
\newtheorem{exas}[thm]{Examples}

\numberwithin{equation}{section}

\newcommand{\hah}{\hat h}
\newcommand{\hag}{\hat g}
\newcommand{\ol}{\bar}
\newcommand{\tq}{\tilde q}
\newcommand{\Ra}{\hbox{$\quad\Rightarrow\quad$}}
\newcommand{\ext}{\raise.5pt\hbox{$\bigwedge$}\kern-1pt}
\newcommand{\Gr}{\mathbb{G}\mathrm{r}}
\newcommand{\vep}{{\varepsilon}}
\newcommand{\hsr}{\hh\setminus\rr}
\newcommand{\ds}{\displaystyle}
\renewcommand{\Im}{\mathrm{Im}}
\newcommand{\Id}{\mathrm{id}}
\renewcommand{\Re}{\mathrm{Re}}
\newcommand{\cp}{\mathbb{CP}}
\newcommand{\hp}{\mathbb{HP}}
\newcommand{\hh}{\mathbb{H}}
\newcommand{\cc}{\mathbb{C}}
\renewcommand{\ll}{\mathbb{\ell}}
\newcommand{\mm}{m}
\newcommand{\nn}{\mathbb{N}}
\newcommand{\rr}{\mathbb{R}}
\newcommand{\zz}{\mathbb{Z}}
\newcommand{\s}{\mathbb{S}}
\newcommand{\jj}{\mathbb{J}}
\newcommand{\sK}{\mathscr{K}}
\newcommand{\sO}{\mathscr{O}}
\newcommand{\sQ}{\mathscr{Q}}
\newcommand{\sS}{\mathscr{S}}

\newcommand{\frs}[2]{\hbox{\large$\frac{#1}{#2}$}}
\renewcommand{\ge}{\geqslant}
\renewcommand{\le}{\leqslant}
\newcommand{\vv}{\mathbf{v}}
\newcommand{\ww}{\mathbf{w}}
\newcommand{\tF}{\mathcal{F}}

\newcommand{\pgamma}{\gamma}
\newcommand{\pGamma}{\Gamma}
\newcommand{\PGamma}{\hbox{\large$\mathfrak{G}$}}

\begin{document}

\renewcommand{\oddsidemargin}{20pt}
\renewcommand{\evensidemargin}{20pt}

\titlerunning{Twistor transforms of quaternionic functions}

\title{Twistor transforms of quaternionic functions\\
and orthogonal complex structures}

\author{Graziano Gentili \and Simon Salamon \and Caterina Stoppato}

\date{}

\maketitle

\address{G.~Gentili: Universit\`a di Firenze, Dipartimento di Matematica e Informatica``U. Dini'', Viale Morgagni 67/A, 50134 Firenze, Italy;\\
  \email{gentili@math.unifi.it} 
\and
S.~Salamon: Department of Mathematics, King's College London, Strand, London WC2R 2LS, UK;\\
\email{simon.salamon@kcl.ac.uk}
\and
C.~Stoppato: Istituto Nazionale di Alta Matematica, Unit\`a di Ricerca di Firenze c/o DiMaI ``U. Dini'' Universit\`a di Firenze, Viale Morgagni 67/A, 50134 Firenze, Italy;\\
\email{stoppato@math.unifi.it}}

\subjclass{53C28, 30G35; 53C55, 14J26}

\vspace{-200pt} 
\centerline{\color[rgb]{.5,.5,.5}
\sf To appear in the Journal of the European Mathematical Society}
\vspace{150pt}

\begin{abstract}
  The theory of slice-regular functions of a quaternion variable is
  applied to the study of orthogonal complex structures on domains
  $\Omega$ of $\rr^4$. When $\Omega$ is a symmetric slice domain, the
  twistor transform of such a function is a holomorphic curve in the
  Klein quadric. The case in which $\Omega$ is the complement of a
  parabola is studied in detail and described by a rational quartic
  surface in the twistor space $\cp^3$.
\end{abstract}

\section{Introduction}

It is well known that quaternions can be used to describe orthogonal
complex structures on $4$-dimensional Euclidean space, since their
imaginary units parametrize isomorphisms $\rr^4\cong\cc^2.$ The
resulting theory is invariant by the group $SO(5,1)$ of conformal
automorphisms of $S^4=\hp^1$, locally isomorphic to the group
$SL(2,\hh)$ of linear fractional transformations (see
section~\ref{induced}). Moreover, it crops up naturally in the
description of certain dense open subsets of $\rr^4=\hh$, which
possess complex structures that are induced from algebraic surfaces in
the twistor space $\cp^3$ that fibres over $S^4.$

We begin with the general definition of this class of structures.

\begin{defin}
  Let $(M^4, g)$ be a $4$-dimensional oriented Riemannian manifold. An
  \emph{almost complex structure} is an endomorphism $J : TM \to TM$
  satisfying $J^2 = -\Id$. It is said to be \emph{orthogonal} if it is
  an orthogonal transformation, that is, $g(J\vv,J\ww) = g(\vv,\ww)$
  for every $\vv,\ww \in T_pM$. In addition, we shall assume that such
  a $J$ preserves the orientation of $M$. An orthogonal almost complex
  structure is said to be an \emph{orthogonal complex structure},
  abbreviated to \emph{OCS}, if $J$ is integrable.
\end{defin}

Every element $I$ in the $2$-sphere
\begin{equation}\label{sph}
\s = \{q\in\hh:q^2 = -1\}
\end{equation} 
of imaginary unit quaternions can be regarded as a complex structure
on $\hh$ in a tautological fashion as follows. After identifying each
tangent space $T_p\hh$ with $\hh$ itself, we define the complex
structure by left multiplication by $I$, i.e.\ $J_p\vv=I\vv$ for all
$\vv\in T_p\hh\cong\hh$. Such a complex structure is called
\emph{constant} and it is clearly orthogonal: if we choose $I'\in \s$
such that $I\perp I'$ then $J_p$ is associated to the matrix
$$
\left(\begin{array}{cccc}
0 & -1 & 0 & 0\\
1 & 0 & 0 & 0\\
0 & 0 & 0 & -1\\
0 & 0 & 1 & 0
\end{array}\right)
$$ with respect to the basis $1,I,I',II'$.

Any OCS defined globally on $\hh\lower10pt\hbox{}$ is known to be
constant; see \cite{wood}. More generally,

\begin{thm} [{\rm\cite{viaclovsky}}] \label{Hau}
Let $J$ be an OCS of class $C^1$ on $\rr^4 \setminus \Lambda$, where
$\Lambda$ is a closed set having zero $1$-dimensional Hausdorff
measure: $\mathcal{H}^1(\Lambda) = 0$. Then either $J$ is constant or
$J$ can be maximally extended to the complement $\rr^4\setminus \{p\}$
of a point. In both cases, $J$ is the push-forward of the standard OCS
on $\rr^4$ under a conformal transformation.
\vphantom{\lower4pt\hbox{$p$}}
\end{thm}

If the real axis $\rr$ is removed from $\hh$ then a totally different
structure $\jj$ can be constructed, thanks to the property of $\hh$
that we are about to describe. If, for all $I\in\s$, we set $L_I = \rr
+ I \rr$, then $L_I \cong \cc$ and
$$\hh = \bigcup_{I \in \s} L_I,\qquad \bigcap_{I \in \s} L_I = \rr.$$ In
fact, every $q \in \hh \setminus \rr$ can be uniquely expressed in
terms of its real part $\Re(q)$ and imaginary part $\Im(q) = q-\Re(q)$ as
$q= \Re(q)+I_q |\Im(q)|$ where
$$I_q = \frac{\Im(q)}{|\Im(q)|} \in \s.$$
For all $q \in \hh \setminus \rr$ and for all $\vv\in T_q(\hh \setminus \rr) \cong
\hh$, we now define $\jj_q\vv = I_q\vv$. 

Observe that $\jj$ is integrable. Indeed, when we express
\begin{equation}\label{product}
\hh \setminus \rr\ =\ \rr + \s \cdot
\rr^+\ \cong\ \mathbb{CP}^1\times \cc^+,
\end{equation}
as the product of the Riemann sphere and the upper half-plane
$\cc^+=\{x+iy \in \cc: y>0\}$, then $\jj$ is none other than the
product complex structure on the two factors. Furthermore, $\jj$ is
orthogonal: this fact can be proven reasoning pointwise.

In general, an orthogonal almost complex structure $J$ defined on an open subset
$\Omega$ of $\rr^4$ is simply a map $\Omega\to\s$. It is well known
that $J$ is integrable if and only if this map is holomorphic as a
mapping from the almost-complex manifold $(\Omega,J)$ to the Riemann
sphere (provided the orientation of the latter is chosen correctly,
see \cite{EellsSalamon}). It follows that $J$ is integrable if and
only if the associated section over $\Omega$ of the trivial bundle
$\rr^4\times\s$ is holomorphic for the natural twisted complex
structure on this product. This complex manifold can be compactified
(by adding a projective line $\cp^1$) to the ``twistor space''
$\cp^3$, which fibres over $\s^4=\rr^4\cup\{\infty\}$. In practice
then, an OCS $J$ on an open subset $\Omega$ of $S^4$ arises from a
\emph{holomorphic} submanifold of $\cp^3$
\cite{Salamon1985,deBartolomeisNannicini}. This tautological principle
underlies the proof of Theorem~\ref{Hau}, as well as our study of the
properties of $\jj$.
In section~\ref{induced}, we shall write down complex coordinates for
the OCS $\jj$, and use these in section~\ref{twist} to show that $\jj$
arises from a quadric in the twistor space $\cp^3$ (see
proposition~\ref{quadric}). This fact was the basis of the next
result.

\begin{thm} [{\rm\cite{viaclovsky}}] \label{thmjj}
Let $J$ be an OCS of class $C^1$ on $\rr^4\setminus \Lambda$ where
$\Lambda$ is a round circle or a straight line, and assume that $J$ is
not conformally equivalent to a constant OCS. Then $J$ is unique up to
sign, and $\rr^4\setminus \Lambda$ is a maximal domain of definition
for $J$.
\end{thm}

\noindent Thus $\jj$ and $-\jj$ are the only non-constant OCS's on
$\hh \setminus \rr$. One might hope to achieve similar
characterisations of more complicated complex structures.\medbreak

For a closed subset $\Lambda$ of $\hh$ other than those described in
the previous theorems, it is in general an open question as to whether
OCS's exist on $\hh\setminus \Lambda$, let alone the task of
attempting their classification. No systematic methods were known for
addressing this problem, other than attempting to identify graphs in
the twistor space $\cp^3$ manufactured from some obvious algebraic
subsets.

In the present work, we solve such a problem working directly on $\hh$
by exhibiting a rather surprising link with the class of quaternionic
functions recently introduced in \cite{advances}. This theory is based
on the following notion of regularity.

\begin{defin}\label{reg}
  Let $\Omega$ be a domain in $\hh$ and let $f:\Omega \to \hh.$ For
  all $I \in \s$, let us denote $\Omega_I = \Omega \cap L_I$ and $f_I
  = f_{|_{\Omega_I}}$.  The function $f$ is called \emph{slice
    regular} if, for all $I \in \s$, the restriction $f_I$ is
  holomorphic, i.e.\ $\bar \partial_I f : \Omega_I \to \hh$ defined by
\[\bar \partial_I f (x+Iy) = {\textstyle\frac12}\!\left(
\frac{\partial}{\partial x}+I\frac{\partial}{\partial y} \right) f_I
(x+Iy).\]
vanishes identically.
\end{defin}

\noindent Throughout this paper we use ``regular'' to stand for
``slice regular'', see Remark~\ref{JQ}. Provided $\Omega$ satisfies
the additional properties of Definition~\ref{SSD}, regular functions
defined on $\Omega$ are necessarily real analytic, and they have some
nice properties in common with holomorphic functions of one complex
variable. We shall explain this in
section~\ref{preliminaries}.\smallbreak

Section~\ref{diff} is devoted to a detailed study of points where the
differential of a regular function $f$ is not invertible, and takes
account of a phenomenon whereby $f$ can be constant on isolated
$2$-spheres. Theorem~\ref{fully} fully describes the behaviour of $f$
in such cases.
The link between regular functions and OCS's, central to this paper,
is established in section~\ref{induced}. We discuss various
automorphisms of $\jj$ and the domain $\hsr$, and enumerate more
general classes of maps that preserve orthogonality.

Section~\ref{twist} contains a twistorial interpretation of the
quaternionic analysis, in which one passes from $\hp^1$ to $\cp^3$ in
order to render the theory complex algebraic. Theorem~\ref{lift}
asserts that a regular function $f$ lifts to a holomorphic mapping
from an open set of a quadric (the graph of $\jj$) in $\cp^3$ onto
another complex surface in $\cp^3$. This enables us to associate to
$f$ a holomorphic curve $\tF$ in the Klein quadric in $\cp^5$,
satisfying a reality condition, and the construction can be inverted
(see Theorem~\ref{transform}). The curve $\tF$ is the transform of the
title, and it encodes the action of $f$ on the family of spheres
$x+y\s$ symmetric with respect to the real axis. Properties of regular
functions can then be interpreted in natural holomorphic terms.

We illustrate our approach by describing OCS's induced by mapping the
real axis of $\hh$ onto a parabola $\pgamma$. Because of the
importance of this example, it effectively occupies the last two
sections. Section~\ref{removing} exploits the theory of the quadratic
polynomial $q^2+qi+c$ over $\hh$ and properties of its roots
(generically, there are two). In section~\ref{squartic}, we show that
the induced complex structures arise from a quartic surface $\sK$ in
$\cp^3$, whose singular points form three lines. It is one of a number
of quartic scrolls whose classification (described in \cite{dolg} and
references therein) might be relevant for the treatment of analogous
orthogonal complex structures. The map $f$ determines a resolution of
$\sK$ via Theorem~\ref{bir}, and the ``twistor geometry of conics''
could be further developed using Lie sphere techniques
\cite{shapirog}.

In \cite[Theorem 14.4 of v1]{viaclovsky}, it was shown that a
(non-singular) algebraic subvariety of $\cp^3$ of degree $d>2$ cannot
define a single-valued OCS unless one removes a set of real dimension
at least 3. By applying similar methods, we prove the non-existence of
an OCS on the complement $\hh\setminus \pgamma$ of the parabola (see
Theorem~\ref{nonex}). Indeed, in the geometry of $\hh\setminus\pgamma$
that we describe in this paper, one needs to remove a solid paraboloid
$\PGamma$ as well as $\pgamma$ in order to define an OCS.

\section{Factorisation of regular functions}\label{preliminaries}

Definition~\ref{reg} has its origin in work of Cullen
  \cite{cullen}. The class of regular functions includes every
  polynomial or power series of the form
$$f(q) = \sum_{n\in\nn} q^n a_n$$ on its ball of convergence $B(0,R) =
  \{q \in \hh : |q| <R\}$. Throughout this paper, $\nn$ denotes the
  set of natural numbers $0,1,2,\ldots$

Conversely, a regular function on a ball $B(0,R)$ is the sum of a
power series of the same type. Furthermore, the set of power series
converging in $B(0,R)$ is a ring with the usual addition operation $+$
and the multiplicative operation $*$ defined by
\begin{equation}\label{prodottostar} 
  \bigg(\sum_{n \in \nn} q^n a_n
  \bigg)*\bigg( \sum_{n \in \nn} q^n b_n \bigg) = \sum_{n \in \nn} q^n
  \sum_{k = 0}^n a_k b_{n-k}.
\end{equation} 
Recently, \cite{advancesrevised} identified a larger class of
domains on which the theory of regular functions is natural and not
limited to quaternionic power series.

\begin{defin}\label{SSD}
A domain $\Omega \subseteq \hh$ is called a \emph{slice domain} if
$\Omega_I = \Omega \cap L_I$ is an open connected subset of $L_I \cong
\cc$ for all $I \in \s$, and the intersection of $\Omega$ with the
real axis is non-empty. A slice domain $\Omega$ is a \emph{symmetric
  slice domain} if it is symmetric with respect to the real axis.
\end{defin}

\noindent The word \emph{symmetric} will refer to axial symmetry with
respect to $\rr$ throughout the paper. Recall the 2-sphere $\s$
defined in \eqref{sph}. If $\Omega$ is a symmetric slice domain then
$\Omega \setminus \rr$ can be expressed as
\begin{equation}\label{domain}
\Omega \setminus \rr =\bigcup_{x+iy\in V}(x+y\s)\ \cong\ \cp^1\times V
\end{equation}
relative to \eqref{product}, where $x+y\s = \{x+Iy : I \in \s\}$ and
$V$ is an open connected subset of $\cc^+$.\medbreak

As explained in \cite{advancesrevised}, regular functions on a
symmetric slice domain admit an identity principle, and they are
$C^\infty$. These properties are not granted for other types of
domains, see e.g.\ \cite{advancesrevised}.  Additionally, it was
proven in the same paper that the multiplicative operation $*$ can be
extended so that the set of regular functions on a symmetric slice
domain $\Omega$ is endowed with a ring structure. It turns out that
this algebraic structure is strictly related to the distribution of
the zeros (see \cite{zeros,zerosopen,milan}). In fact, the zero set of
a regular function on a symmetric slice domain consists of isolated
points or isolated $2$-spheres of type $x+y\s$, and can be studied as
follows.

\begin{thm}\label{factorization}
Let $f\not\equiv0$ be a regular function on a symmetric slice domain
$\Omega$, and let $x+y\s \subset\Omega$. There exist $m \in \nn$ and a
regular function $\tilde f:\Omega \to \hh$ not identically zero in
$x+y\s$ such that
\begin{equation*}
f(q) = \big[(q-x)^2+y^2\big]^m \tilde f(q).
\end{equation*}
If $\tilde f$ has a zero $p_1 \in x+y\s$ then such a zero is unique
and there exist $n \in \nn$, $p_2,...,p_n \in x+y\s$ (with $p_i \neq
\bar p_{i+1}$ for all $i \in \{1,\ldots,n-1\}$) such that
\begin{equation*}
\tilde f(q) = (q-p_1)*(q-p_2)*\cdots*(q-p_n)*g(q)
\end{equation*}
for some regular function $g:\Omega \to \hh$ which does not have zeros in $x+y\s$.
\end{thm}

\bigbreak

\begin{defin}
In the situation of Theorem \ref{factorization}, $f$ is said to have
\emph{spherical multiplicity} $2m$ at $x+y\s$ and \emph{isolated
  multiplicity} $n$ at $p_1$. Finally, the \emph{total multiplicity}
of $x+y\s$ for $f$ is defined as the sum $2m+n$.
\end{defin}

The meaning of these notions is clarified by the case of quadratic
polynomials, which will also prove useful in the sequel.

\begin{exa}\label{roots}
Let $\alpha,\beta\in \hh$ and set $P(q)=(q-\alpha)*(q-\beta)$.
\begin{enumerate}
\item If $\beta$ is not in the same sphere $x+y\s$ as $\alpha$ then $P$
  has two distinct roots, $\alpha$ and 
$(\alpha -\bar\beta)\beta (\alpha -\bar \beta)^{-1}$, 
each having isolated multiplicity $1$.
\item If $\alpha,\beta$ lie in the same sphere $x+y\s$ but $\alpha
  \neq \bar \beta$ then $\alpha=(\alpha -\bar \beta) \beta (\alpha
  -\bar \beta)^{-1}$ and the unique root $\alpha$ has isolated
  multiplicity $2$.
\item Finally, if $\alpha=\bar\beta\in x+y\s$ then the zero set of
  $P$ is $x+y\s$, which has spherical (and total) multiplicity $2$.
\end{enumerate}
Observe that, in 1., the second root $(\alpha -\bar\beta)\beta
  (\alpha -\bar \beta)^{-1}$ does not necessarily equal $\beta$ but it does lie in
  the same sphere as $\beta$.
\end{exa}

\begin{rem}\label{JQ}  The class characterised by
  Definition~\ref{reg} differs significantly from other classes of
  quaternionic functions introduced during the last century, in
  particular that arising from solutions of the Fueter equation,
  relevant to the study of maps between hypercomplex manifolds. (See
  \cite{chenli,librodaniele,sudbery} and references therein.) The
  class of Fueter regular functions has the advantage that it has an
  interpretation on any quaternionic manifold \cite{ssalamon}. It has
  a subtle but deep algebraic structure that was identified by Joyce
  \cite{joyce}, whose work was interpreted by Quillen
  \cite{quillen}. But the latter is less relevant to the theory of
  quaternionic polynomials, application of which is really the subject
  of this paper.
\end{rem} 

\section{The differential of a regular function}\label{diff}

In this section, we will prove some new results concerning the (real)
differential $f_*$ of a regular function $f$. We first recall two
theorems proven in \cite{expansion}.

\begin{thm}\label{series}
Let $f$ be a regular function on a symmetric slice domain $\Omega$,
and let $x_0,y_0 \in \rr$ be such that $x_0+y_0\s \subset \Omega$. For
all $q_0 \in x_0+y_0\s$, there exists a sequence $\{A_{n}\}_{n \in
  \nn}\subset \hh$ such that
\begin{equation}\label{expansion2}
f(q) = \sum_{n \in \nn}\big[(q-x_0)^2+y_0^2\big]^n\,[A_{2n} + (q-q_0)A_{2n+1}]
\end{equation}
in every
$U(x_0+y_0\s,R)=\{q\in\hh:|(q-x_0)^2+y_0^2|<R^2\}\subseteq\Omega.$
\end{thm}

Let $q_0=x_0+Iy_0$. We now see that regular functions are
  not merely holomorphic in the plane $L_I$, but that they have an
  analogous derivative perpendicular to $L_I$. The resulting formulae
  involve the first two coefficients (after $A_0$).

\begin{thm}
Let $f$ be a regular function on a symmetric slice domain $\Omega$,
let $q_0 \in \Omega$ and let $A_n$ be the coefficients of the
expansion \eqref{expansion2}. Then for all $v \in \hh$ with $|v|=1$
the derivative of $f$ along $v$ can be computed at $q_0$ as
\begin{equation}\label{differential}
\lim_{t\to 0}\frac{f(q_0+tv)-f(q_0)}{t} = vA_1 + (q_0v-v\bar q_0)A_2.
\end{equation}
\end{thm}
\noindent For more details, we refer the reader to
  \cite{expansion}.\smallbreak

If we identify $T_{q_0}\hh$ with $\hh=L_{I} \oplus L_{I}^\perp$ then
for all $u \in L_I$ and $w \in L_I^\perp$,
\begin{equation}\label{LR}
(f_*)_{q_0}(u+w)=u (A_1 + 2\Im(q_0) A_2) + w A_1.
\end{equation}
In other words, {\it the differential $f_*$ at $q_0$ acts by right
  multiplication by $A_1 + 2\Im(q_0) A_2$ on $L_I$ and by right
  multiplication by $A_1$ on $L_{I}^\perp$.} 

We will now make use of \eqref{LR} to investigate the rank of $f_*$.

\begin{prop}\label{classification}
Let $f$ be a regular function on a symmetric slice domain $\Omega$,
let $q_0 = x_0+Iy_0 \in \Omega \setminus \rr$ and let $A_n$ be the
coefficients of the expansion \eqref{expansion2}.
\begin{itemize}
\item If $A_1=0$ then $f_*$ has rank $2$ at $q_0$ if $A_2 \ne0$, rank $0$ if
  $A_2 = 0$.
\item If $A_1 \ne0$ then either $f_*$ is invertible at $q_0$ or it has 
rank $2$ at $q_0$; the latter happens if and only if $1+2
  \Im(q_0) A_2A_1^{-1} \in L_{I}^\perp$.
\end{itemize}
Finally, for all $x_0 \in \Omega \cap \rr$, $f_*$ is invertible at
$x_0$ if $A_1\neq 0$; it has rank $0$ at $x_0$ if $A_1=0$.
\end{prop}

\begin{proof}
If $A_1=0$ then $(f_*)_{q_0}(u+w)=u\,2\Im(q_0) A_2$ for all $u \in L_I,
w \in L_I^\perp$. Hence the kernel of $(f_*)_{q_0}$ is $L_I^\perp$ if
$A_2 \neq 0$, the whole space $\hh$ if $A_2 = 0$.

Let us now turn to the case $A_1 \ne0$ and observe that for all $u
\in L_I$ and $w \in L_I^\perp$,
$$(f_*)_{q_0}(u+w)=\left[u(1+2\Im(q_0) A_2 A_1^{-1})+w\right]\!A_1.$$
The differential $(f_*)_{q_0}$ is not invertible if and only if $1 +
2\Im(q_0) A_2 A_1^{-1} =p \in L_I^\perp$. In this case, if $p=0$ then
the kernel of $(f_*)_{q_0}$ is $L_I$; if $p \neq0$ then the kernel is
the $2$-plane of vectors $-wp^{-1} +w$ for $w \in L_I^\perp$.

The last statement is proved by observing that if $x_0 \in \Omega \cap
\rr$ then $(f_*)_{x_0}v = v A_1$ for all $v \in \hh$.
\end{proof}

Now let us study in detail the set
$$N_f = \{q \in \Omega : f_*\mathrm{\ is\ not\ invertible\ at\ }
q\},$$ which we may call the \emph{singular set} of $f$.  In this
study, we will make use of the following specific subset of $N_f$ (see
\cite{open,zerosopen}).

\begin{defin}
Let $\Omega$ be a symmetric slice domain and let $f : \Omega \to \hh$
be a regular function. The \emph{degenerate set} of $f$ is the union
$D_{f}$ of the 2-spheres $x+y\s$ (with $y \ne0$) such that
$f_{|_{x+y\s}}$ is constant.
\end{defin}

The following observation will also prove useful (see \cite{poincare}).

\begin{rem}\label{stereo}
For all $q_0 = x_0+Iy_0 \in \hh \setminus \rr$, setting $\Psi(q) =
(q-q_0)(q-\bar q_0)^{-1}$ defines a stereographic projection of
$x_0+y_0\s$ onto the plane $L_{I}^\perp$ from the point $\bar q_0$.
\end{rem}

Indeed, if we choose $J \in \s$ with $J\perp I$ and set $K = IJ$ then
for all $q=x_0+Ly_0$ with $L=tI+uJ+vK \in \s$ we have $\Psi(q) =
(L-I)(L+I)^{-1} = \frac{u+Iv}{1+t} K$ and $(\rr+I\rr)K = L_{I}^\perp$.
We are now in a position to characterise the points of the singular
set in algebraic terms.

\begin{prop}\label{characterization}
Let $f$ be a regular function on a symmetric slice domain $\Omega$,
let $q_0 = x_0+Iy_0 \in \Omega$. Then $f_*$ is not invertible at $q_0$
if, and only if, there exist $\widetilde q_0 \in x_0+y_0\s$ and a
regular function $g: \Omega \to \hh$ such that
\begin{equation}\label{algebraic}
f(q) = f(q_0) + (q-q_0)*(q-\widetilde q_0)*g(q).
\end{equation}
Equivalently, $f_*$ is not invertible at $q_0$ if, and only if, $f-f(q_0)$ 
has total multiplicity $n\ge2$ at $x_0+y_0\s$. Moreover,
$q_0$ belongs to the degenerate set $D_f$ if, and only if, it belongs
to the singular set $N_f$ and there exists a regular $g: \Omega \to
\hh$ such that equation \eqref{algebraic} holds for $\widetilde
q_0=\bar q_0$. The latter is equivalent to saying that $f-f(q_0)$ has
spherical multiplicity $n\ge2$ at $x_0+y_0\s$.
\end{prop}

\begin{proof}
If $q_0 \in \Omega \setminus \rr$ then it belongs to $D_f$ if and only
if, $f$ is constant on the $2$-sphere $x_0+y_0\s$, i.e. there exists a
regular function $g: \Omega \to \hh$ such that
$$f(q) = f(q_0) + [(q-x_0)^2+y_0^2]*g(q) = f(q_0) +(q-q_0)*(q-\bar
q_0)*g(q).$$ This happens if and only if the coefficient $A_1$ in the
expansion \eqref{expansion2} vanishes.

Now let us turn to the case $q_0 \in \Omega \setminus \rr$, $q_0 \not
\in D_f$. By proposition \ref{classification}, $q_0 \in N_f$ if and
only if $1 + 2\Im(q_0) A_2 A_1^{-1} = p \in L_I^\perp$. Thanks to the
previous remark, $p \in L_I^\perp$ if, and only if, there exists
$\widetilde q_0 \in (x_0+y_0\s)\setminus \{\bar q_0\}$ such that $p =
(\widetilde q_0-q_0)(\widetilde q_0-\bar q_0)^{-1}$. The last formula
is equivalent to
$$2\Im(q_0) A_2 A_1^{-1} = (\widetilde q_0-q_0 - \widetilde q_0+\bar
q_0) (\widetilde q_0-\bar q_0)^{-1} = -2\Im(q_0) (\widetilde q_0-\bar
q_0)^{-1}$$ i.e. $A_1 = (\bar q_0 - \widetilde q_0) A_2$. Finally,
the last equality is equivalent to
\[\begin{array}{rl}
f(q) \!\!&= A_0+ (q-q_0) (\bar q_0 - \widetilde q_0) A_2 +
\big[(q-x_0)^2+y_0^2\big] A_2 \\
&\hskip200pt +\big[(q-x_0)^2+y_0^2\big](q-q_0)*h(q)\\[3pt]
& =\ f(q_0) + (q-q_0)*\big[\bar q_0 - \widetilde q_0 + q-\bar q_0\big] A_2 
+ (q-q_0)*\big[(q-x_0)^2+y_0^2\big]*h(q)\\[3pt]
&=\ f(q_0) + (q-q_0)*(q- \widetilde q_0) A_2 + 
(q-q_0)*(q- \widetilde q_0)*(q- \overline{\widetilde q_0})*h(q)\\[3pt]
&=\ f(q_0) + (q-q_0)*(q- \widetilde q_0) *\big[A_2 
+(q- \overline{\widetilde q_0})*h(q)\big],
\end{array}\]
for some regular $h :\Omega \to \hh$.

To conclude, we observe that if $x_0 \in \Omega \cap \rr$ then $A_1
=0$ if and only if
$$f(q) = f(q_0) + (q-x_0)^2*g(q) = f(q_0) +(q-x_0)*(q-x_0)*g(q)$$
for some regular function $g: \Omega \to \hh$.
\end{proof}

We end this section with a complete characterization of the singular
set of $f$, proving in particular that it is empty when $f$ is an injective
function. The following two results, from \cite{open,zerosopen}, will
prove useful to this end.

\begin{thm}
Let $\Omega$ be a symmetric slice domain and let $f : \Omega \to \hh$
be a non-constant regular function. The degenerate set $D_f$ is closed
in $\Omega\setminus \rr$ and it has empty interior.
\end{thm}

\begin{thm}[Open Mapping Theorem]\label{open}
Let $f$ be a regular function on a symmetric slice domain $\Omega$ and
let $D_f$ be its degenerate set. Then $f:\Omega \setminus
\ol{D}_f \to \hh$ is open.
\end{thm}

We are now in a position to fully characterise the singular set of $f$.

\begin{thm}\label{fully}
Let $\Omega$ be a symmetric slice domain and let $f : \Omega \to \hh$
be a non-constant regular function. Then its singular set $N_f$ has
empty interior. 
Moreover, for a fixed $q_0 = x_0+Iy_0\in N_f$, let $n>1$ 
be the total multiplicity of $f-f(q_0)$ at $x_0+y_0\s$. Then there
exist a neighbourhood $U$ of $q_0$ and a neighbourhood $T$ of
$x_0+y_0 \s$ such that, for all $q_1 \in U$, the sum of the total multiplicities 
of the zeros of $f-f(q_1)$ in $T$ equals $n$; in particular, for all 
$q_1 \in U\setminus N_f$ the preimage of $f(q_1)$ includes at least two 
distinct points of $T$.
\end{thm}

\begin{proof}
Since $f$ is not constant, the singular set $N_f$ has empty interior
if and only if $N_f \setminus D_f$ does. By proposition
\ref{classification}, the rank of $f_*$ equals $2$ at each point of
$N_f \setminus D_f$. If $N_f \setminus D_f$ contained a non empty open
set $A$, its image $f(A)$ could not be open by the constant rank
theorem, and the open mapping theorem \ref{open} would be
contradicted.

In order to prove our second statement, let us introduce the notation
$f_{q_1}(q) = f(q) - f(q_1)$ for $q_1 \in \Omega$. By proposition
\ref{characterization}, if $q_0 = x_0+Iy_0 \in N_f$ then the total
multiplicity $n$ of $f_{q_0}$ at $x_0+y_0\s$ is strictly greater than $1$. 
We will now make use of the operation of symmetrization defined in
\cite{advancesrevised} setting $f^s = f*f^c = f^c*f$ where $f^c$ is
defined by formula $$\bigg(\sum_{n \in \nn} q^n a_n\bigg)^{\!c} = 
\sum_{n \in \nn} q^n \bar a_n$$ on power series, and appropriately
extended to regular functions on symmetric slice domains. By direct
computation, the total multiplicity of each $2$-sphere $x+y\s$ for
$f^s$ is twice its total multiplicity for $f$.  Hence, in our case
$x_0+y_0\s$ has total multiplicity $2n$ for $f_{q_0}^s$, i.e., there
exists a regular $h: \Omega \to \hh$ having no zeros in $x_0+y_0\s$
such that $f_{q_0}^s(q) = [(q-x_0)^2+y_0^2]^{n} h(q)$ and
$$f_{q_0}^s(z) = \big[(z-x_0)^2+y_0^2\big]^n\,h(z) = 
(z-q_0)^n(z-\bar q_0)^nh(z)$$ for all $z \in \Omega_I$. Furthermore,
$f_{q_0}^s(\Omega_I) \subseteq L_I$ as explained in \cite{zerosopen},
so that the restriction of $f_{q_0}^s$ to $\Omega_I$ can be viewed as
a holomorphic complex function.  Let us choose an open $2$-disc
$\Delta$ centred at $q_0$ such that $f_{q_0}^s$ has no zeros in
$\ol{\Delta} \setminus \{q_0\}$, with $\Delta$ strictly included
both in $\Omega_I$ and in the half-plane of $L_I$ that contains
$q_0$. If we denote by $F_{q_1}$ the restriction of $f_{q_1}^s$ to
$\Delta$ then $q_1 \mapsto F_{q_1}$ is continuous in the topology of
compact uniform convergence and $F_{q_0}$ has a zero of multiplicity
$n$ at $q_0$ (and no other zero). We claim that there exists a
neighbourhood $U$ of $q_0$ such that for all $q_1 \in U$ the sum of the
multiplicities of the zeros of $F_{q_1}$ in $\Delta$ is $n$: if this
were not the case, it would be possible to construct a sequence
$\{q_k\}_{k \in \nn}$ converging to $q_0$ such that $\{F_{q_k}\}_{k
  \in \nn}$ contradicted Hurwitz's theorem (in the version of
\cite{libroconway}). Now let $R>0$ be the radius of $\Delta$ and let
$$T=T(x_0+y_0 \s, R) = \{x+Jy: |x-x_0|^2+|y-y_0|^2 < R^2, J \in \s\}$$
be its symmetric completion. Then for all $q_1 \in U$, $2n$ equals the
sum of the spherical multiplicities of the zeros of $f_{q_1}^s$ in
$T$.  Hence, there exist points $q_2,\ldots,q_n \in T$ and a
function $h: \Omega \to \hh$ having no zeros in $T$ such that
$$f_{q_1}(q) = (q-q_1)*\cdots*(q-q_n)*h(q).$$ Now let us suppose $q_1
= x_1+Jy_1 \in U\setminus N_f$. Then $f_*$ is invertible at $q_1$ and,
by proposition \ref{characterization}, the point $q_2$ cannot 
belong to the sphere $x_1+y_1\s$. Hence, $f_{q_1}(q)$ has at least $2$ 
zeros in $T$ and the preimage of $f(q_1)$ via $f$ intersects $T$ at least 
twice.
\end{proof}

\begin{cor}\label{injective}
Let $\Omega$ be a symmetric slice domain and let $f:\Omega\to\hh$ be a
regular function. If $f$ is injective then its singular set $N_f$ is
empty.
\end{cor}

\section{Induced complex structures}\label{induced}

This section starts by explaining the link between the orthogonal
complex structure $\jj$ and the theory of regular functions. The idea
is simple, namely that a regular function maps $\jj$ locally to
another \emph{orthogonal} complex structure. The main formula,
\eqref{key} below, expresses this fact even more simply, but its
consequences are significant. 

The real differential of an injective regular function is invertible
at all points, thanks to Corollary \ref{injective}. This allows us to
push-forward the complex structure $\jj$ we defined on $\hh \setminus
\rr$; we recall that for all $q \in \hh \setminus \rr$ and for all
$\vv\in T_q\hh \cong \hh$, $$\jj_q\vv = I_q\vv$$ where
$$I_q = \frac{\Im(q)}{|\Im(q)|} \in \s.$$

\begin{defin}
Let $\Omega$ be a symmetric slice domain, and let $f: \Omega \to \hh$
be an injective regular function. The \emph{induced structure} on
$f(\Omega \setminus \rr)$ is the push-forward
\begin{equation*}
\jj^f = f_*\,\jj\,(f_*)^{-1}.
\end{equation*}
\end{defin}

We now find an explicit expression for this induced structure, 
thereby proving that it is orthogonal.

\begin{prop}\label{induce}
Let $\Omega$ be a symmetric slice domain, and let $f: \Omega \to \hh$
be an injective regular function. Then
\begin{equation}\label{key}
\jj^f_{f(q)}\vv = I_q\vv 
\end{equation}
for all $q\in \Omega \setminus \rr$ and $\vv\in T_{f(q)} f(\Omega
\setminus \rr) \cong \hh$. As a consequence, $\jj^f$ is an OCS on
$f(\Omega \setminus \rr)$.
\end{prop}

\begin{proof}
If $\ww = (f_*)^{-1}_{f(q)}\vv$ then by direct computation
$$\jj^f_{f(q)}\vv = (f_*)_q \jj_q (f_*)^{-1}_{f(q)}\vv= (f_*)_q
\jj_q\ww =(f_*)_q I_q\ww.$$ Thanks to formula \eqref{differential} and
to the fact that $q$ and $I_q$ commute,
\[\begin{array}{rcl}
(f_*)_q I_q\ww &=&  I_q\ww A_1 + (q I_q\ww - I_q\ww \bar q)A_2\\[3pt]
 &=& I_q (\ww A_1 + (q\ww -\ww \bar q)A_2)\\[3pt]
&=& I_q (f_*)_q\ww\\[3pt]
&=& I_q\vv,
\end{array}\]
as desired.

Now, $f$ is an injective holomorphic map from the complex manifold
$(\Omega \setminus \rr, \jj)$ to the almost-complex manifold
$(f(\Omega\setminus \rr), \jj^f)$. Hence, $\jj^f$ can be viewed as a
holomorphic map from the almost-complex manifold $(f(\Omega\setminus
\rr), \jj^f)$ to $\s$, a fact which implies that $\jj^f$ itself is an
OCS (see the remarks preceding Theorem~\ref{thmjj}).
\end{proof}

\noindent It is important to note that \eqref{key} asserts that $f$ is
not in general $\jj$-holomorphic. This is because $I_q$ does not in
general coincide with $I_{f(q)}$ and, in other words, $f$ does not
preserve the subspace $\cc\cong L_{I_q}\subset\hh$ defined at every 
point $q$ of its domain. Next, we shall however discuss examples in 
which $f$ \emph{is} $\jj$-holomorphic.\medbreak

At this juncture, we need to introduce the quaternionic projective
line $\hp^1$. We define this to be the set of equivalence classes
$[q_1,q_2]$, where 
\begin{equation}\label{left}
  [q_1,q_2]=[pq_1,pq_2],\qquad \forall p\in\hh^*.
\end{equation}
The choice of \emph{left} multiplication is dictated by the choice in
Definition~\ref{reg}. We choose to embed $\hh$ as the affine line in
$\hp^1$ by mapping $q\in\hh$ to $[1,q]$, so that $[0,1]$ is the point
of infinity.

Once one identifies $S^4=\hh\cup\{\infty\}$ with $\hp^1$, it is well
known that the group of conformal transformations corresponds to the
group of invertible transformations
\begin{equation}\label{trans}
[q_1,q_2]\mapsto[q_1d+q_2c,\ q_1b+q_2a],\qquad a,b,c,d\in\hh
\end{equation}
(with ordering chosen to match \cite{moebius}).
The invertibility condition is
\[|a|^2|d|^2+|b|^2|c|^2-2\Re(\ol b d\ol c a)\ne0,\]
the left-hand side being the real determinant when
\hbox{\footnotesize$\Big(\!\!\begin{array}{c}a\ \ c\\b\ \ d\end{array}
  \!\!\Big)$} is embedded in $\mathfrak{gl}(4,\rr)$.  Restricting to
the affine line, \eqref{trans} becomes the linear fractional
transformation
\begin{equation}\label{mob}
q\mapsto (qc+d)^{-1}(qa+b).
\end{equation}
The group generated by these transformations is double covered by
$SL(2,\hh)$. Since the condition that a complex structure be
orthogonal depends only on the underlying conformally flat structure
of $\rr^4$, any element of $SL(2,\hh)$ certainly maps $\jj$ to another
\emph{orthogonal} complex structure.

The maps \eqref{mob} are not in general regular, in part because
regularity is not conserved under composition. However an analogous
class of regular M\"obius transformations was defined by the
third author in \cite{moebius}. These coincide with \eqref{mob} when
$c$ and $d$ are real, in which case we see that \eqref{mob} is regular
because $qc+d$ takes values in the complex plane $L_I$ where $I=I_q$
and its reciprocal can be expanded as a power series in $q$ with real
coefficients.

\begin{prop}\label{SO2H}
The subgroup of $SL(2,\hh)$ that maps $\jj$ to itself is 
\[ SO(2,\hh)\cong Sp(1)\times_{\zz_2}SL(2,\rr),\]
consisting of elements of \eqref{mob} such that $a,b,c,d$ are all real
multiples of the same unit quaternion $\vep\in Sp(1)\cong SU(2))$.
\end{prop}

\begin{proof}
Suppose that the bijection $f$ maps $\jj$ to itself. Inherent in this
condition is that $f$ maps the real axis (and so its complement) to
itself. Thus
\[(x\ol c+\ol d)(xa +b)=x^2\ol ca + x(\ol da + \ol cb)+\ol db\in\rr, 
\qquad\forall x\in\rr.\] This obviously holds given the
hypothesis. Conversely, if $\ol ca$, $\ol da+\ol cb$, $\ol db$ are all
real then we can write $a=\alpha c$ and $b=\beta d$ with
$\alpha,\beta\in\rr$. This implies that $\alpha\ol dc+\beta\ol
cd\in\rr$, which (since $\alpha\ne\beta$) forces $\ol dc\in\rr$ and so
$c=\gamma d$ with $\gamma\in\rr$. We conclude by setting $d=\delta\vep$
with $\delta\in\rr$, $\vep\in\hh$ and $|\vep|=1$.
 
Any element of $SO(2,\hh)$ is the composition of a real M\"obius
transformation $f\in SL(2,\rr)$ and the mapping
\begin{equation}\label{vep}
  f_\vep\colon\ q\mapsto\vep^{-1}\!q\kern1pt\vep,
\qquad\vep\in Sp(1).
\end{equation}
As remarked above, the former type preserves $\jj$. The latter fixes
the real axis, so by the identity principle cannot be regular
unless $\vep=\pm1$ and $f_\vep$ is the identity. On the other hand, it
\emph{is} $\jj$-holomorphic. For the differential of \eqref{vep} maps
$\vv\in\hh$ to $\vep^{-1}\!\vv\vep$ and \[ \jj((f_\vep)_*\vv) =
I_{\vep^{-1}\!q\vep}(\vep^{-1} \vv\vep) = 
\vep^{-1}\!I_q\vep(\vep^{-1}\!\vv\vep) =
\vep^{-1}\kern-1pt(I_q\vv)\kern1pt\vep = 
(f_\vep)_*(I_q\vv) = (f_\vep)_*(\jj\vv),\] 
so the induced complex structure at $\vep^{-1}\!q\kern1pt \vep$ coincides
with $\jj$.
\end{proof} 

\noindent The group $SO(2,\hh)$, sometimes denoted $SO^*(4)$, has the
same complexification as the orthogonal group
$SO(4)=Sp(1)\times_{\zz_2}Sp(1)$. It double covers the \emph{isometry}
group of $\hsr\cong\cp^1\times \cc^+$, endowed with the
product of metrics with constant curvature $\pm1$ (recall
\eqref{product}). The representation of $SO(2,\hh)$ on the space of
real symmetric $3\times3$ matrices was the object of study in
\cite[section~4]{viaclovsky}.\medbreak

\begin{rem}\label{five}
The following classes of functions all map the orthogonal complex
structure $\jj$ to another \emph{orthogonal} complex structure:

{\it1.} $SO(2,\hh)/\zz_2$, the group of $\jj$-holomorphic conformal
transformations;

{\it2.} $SL(2,\hh)/\zz_2$, the group of all conformal transformations;

{\it3.} the pseudo-group of all $\jj$-holomorphic regular maps;

{\it4.} the pseudo-group of all $\jj$-holomorphic maps;

{\it5.} the set of all regular maps.

\noindent 
  Recall that $\jj$ is merely the product complex structure relative
  to \eqref{domain}. Note that the group $PGL(2,\cc)$ of automorphisms
  of $\cp^1$ (acting as the identity on $\cc^+$) lies in {\it4} but
  not {\it3}. This is because an element of $PGL(2,\cc)$ is only
  regular if it is the identity (by the remark following
  \eqref{vep}). We shall investigate a mapping properly in class
        {\it5} in section~\ref{removing}.
\end{rem}

In preparation for the next section, we next describe a pair $u,v$ of
holomorphic coordinates for the complex structure $\jj$.

It is well known that conjugation by any quaternion (as in
\eqref{vep}) acts on $\hh=\rr\oplus\rr^3$ as a rotation on the $\rr^3$
summand (and obviously fixes the real part of any element
$q\in\hh$). With this in mind, define
\[ Q_u=1+uj,\qquad u\in\cc,\]
and  consider the mapping
\[ \begin{array}{crcl}
\phi\colon & \cc\times\cc^+ &\longrightarrow& \hh\\[3pt]
& (u,v) &\mapsto& Q_u^{-1}v\kern1pt Q_u.
\end{array}\]
We shall write $v=x+iy$, so $y>0$ and the complex number $v$
represents an element in the upper half plane $\cc^+$.

It follows that $\phi(u,i)$ belongs to the sphere $\s$ of unit
quaternions, and (up to a factor $-i$) it is in fact the inverse of
the stereographic projection $\s\to\cc$ from the point $-i$,
cf.~\eqref{stereo}. To see this, suppose that
\[   I=ai+bj+ck = Q_u^{-1}i\kern1pt Q_u \in \s,\]
so that
\[  (1+uj)(ai+(b+ic)j)=i+iuj,\]
whence
\begin{equation}\label{UBC}
 u = -i\frac{b+ic}{1+a}.
\end{equation}
We can write any $I\in\s\setminus\{-i\}$ as $\phi(u,i)$ for a unique
$u\in\cc$, but of course if we extend the domain of $u$ to $\cp^1$ by
allowing $u$ to assume the value $\infty$, then $\phi$ exhibits the
isomorphism \eqref{product}.

An element $q\in\hsr$ can now be represented by means of
$\phi$:
\begin{equation}\label{uvu}
 q=Q_u^{-1}v\kern1pt Q_u=(1+uj)^{-1}v(1+uj),
\end{equation}
and this makes it easier to analyse maps with $\hsr$ as domain.  For
example, to apply an element of $PGL(2,\cc)$ (mentioned in
remark~\ref{five}) one merely replaces $u$ in \eqref{uvu} by
$(cu+d)/(au+b)$ with $a,b,c,d\in\cc$. The relevance of \eqref{uvu} in
the study of power series can be seen immediately by computing
\[ q^n = (Q_u^{-1}v^nQ_u)^n =Q_u^{-1}v^nQ_u.\]
In the next section, we shall ``eliminate'' the inverse term
$Q_u^{-1}$ by passing to projective coordinates, and interpreting the
above facts twistorially.

\section{Twistor lifts}\label{twist}

Recall the ``leftish'' definition \eqref{left} of the quaternionic
projective line $\hp^1$. It enables us to define the \emph{twistor
  projection}
\begin{equation}\label{pi}
\begin{array}{crcl}
\pi\colon & \cp^3 &\longrightarrow & \hp^1\\[3pt]
& [Z_0,Z_1,Z_2,Z_3] &\mapsto& [Z_0+Z_1j,\ Z_2+Z_3j].
\end{array}
\end{equation}
Observe that, with our choices, changing the complex representative
$(Z_0,Z_1,Z_2,Z_3)\in\cc^4\setminus\{0\}$ will not affect the image in
$\hp^1.$

Embed $\hh$ as the affine line in $\hp^1$ by mapping $q\in\hh$ to
$[1,q]$. Note that $[0,1]$ then corresponds to the point $\infty\in
S^4$, and that
\begin{equation}\label{fibinf}
 \pi^{-1}(\infty) =\{[0,0,Z_2,Z_3]: [Z_2,Z_3] \in\cp^1\}.
\end{equation}
With the notation of the previous section,
\[
 [1,\ q]\ =\ [Q_u,\ v\kern1pt Q_u] = [1+uj,\ v(1+uj)]\\[3pt]
= \pi[1,\,u,\,v,\,uv].
\]
We therefore obtain

\begin{prop}\label{quadric}
The complex manifold $(\hsr,\>\jj)$ is biholomorphic to the
open subset $\sQ^+$ of the quadric
\[\sQ=\{[Z_0,Z_1,Z_2,Z_3]\in\cp^3: Z_0Z_3=Z_1Z_2\}\]
whose elements satisfy at least one of the following conditions:
\begin{itemize}
\item $Z_0\ne0$ and $Z_2/Z_0\in\cc^+$,
\item $Z_1 \neq 0$ and $Z_3/Z_1 \in \cc^+$.
\end{itemize}
\end{prop}

\bigbreak

Consider a quaternionic entire function

\begin{equation}\label{entire}
 f(q)=\sum_{n\in\nn}q^n a_n=\sum_{n\in\nn}q^n(b_n+c_nj),
\end{equation}
where $b_n,c_n\in\cc$. Repeating the calculation above with $q =
Q_u^{-1}v\kern1pt Q_u$, we have
\begin{equation}\label{repeat}
\begin{array}{rcl}\ds
[1,\ f(q)]\ =\ \Big[Q_u,\ \sum_{n\in\nn}v^nQ_ua_n\Big]
&=&\ds
\Big[1+uj,\ \sum_{n\in\nn}v^n(1+uj)(b_n+c_nj)\Big]\\[9pt]
&=&\ds
\Big[1+uj,\ \sum_{n\in\nn}v^n(b_n-u\bar c_n) +
\sum_{n\in\nn}v^n(c_n+u\ol b_n)j\Big]\\[12pt]
&=& \big[1+uj,\ g(v)-u\,\hah(v)+(h(v)+u\,\hag(v))j\big],
\end{array}
\end{equation}
where 
\begin{equation}\label{gbhc}
g(v)=\sum b_nv^n,\qquad h(v)=\sum c_nv^n
\end{equation}
are holomorphic functions, and 
$\hag(z) = \overline{g(\bar z)},\hah(z)=\overline{h(\bar z)}$
are obtained from $g,h$ by Schwarz reflection. 

This shows that the holomorphic mapping $F\colon\sQ^+\to\cp^3$ 
defined by
\begin{equation}\label{F}
F[1,\,u,\,v,\,uv]=\big[1,\ u,\ g(v)-u\,\hah(v),\ h(v)+u\,\hag(v)\big]
\end{equation}
satisfies $\pi\circ F=f\circ\pi$, and is therefore a
\emph{twistor lift} of $f$. 

\begin{rem} The conversion of $f$ into the two holomorphic 
  functions $g, h$ can be regarded as an application of the splitting
  lemma of \cite{advances}: for all $v$ in the plane $L_i = \cc$ we
  have $f(v)=g(v)+h(v)j$. The fact that $F$ preserves $u$ in the
    second slot is another way of expressing the formula \eqref{key}
    in proposition~\ref{induce}, and the linearity in $u$ reflects the
    fact that $F$ is really transforming lines on $\sQ^+$; we explain
    this below. If the coefficients $a_n$ are all real (so $h=0$),
    then $F$ maps $\sQ^+$ into the quadric $\sQ$. In this very special
    case, $f$ will in fact be $\jj$-holomorphic.
\end{rem}

The technique used in \eqref{repeat} can be generalised to other
domains in $\hsr$. This is accomplished by the next result that also
places the argument on a more rigorous footing.

\begin{thm}\label{lift}
Let $\Omega$ be a symmetric slice domain, and let $f:\Omega \to \hh$
be a regular function. Then $f$ admits a twistor lift to $\sO =
\pi^{-1}(\Omega\setminus \rr)\cap \sQ^+$; in other words, there exists
a holomorphic mapping $F\colon \sO\to \cp^3$ of the form
  \eqref{F} such that $\pi\circ F=f\circ\pi$.
\end{thm}

\begin{proof}
  The theorem will be proved if we show that for all
$$U=U(x_0+y_0\s,R) = \{q \in \hh : |(q-x_0)^2+y_0^2| < R^2\} \subseteq
\Omega\setminus \rr$$ the restriction $f_{|_U}$ can be lifted to a
holomorphic $F : \pi^{-1}(U) \cap \sQ^+\to \cp^3$.  By Theorem
\ref{series}, for each $q_0\in x_0+y_0\s$ there exists $\{A_n\}_{n \in
  \nn}\subset \hh$ such that for all $q \in U$
$$f(q) = \sum_{n \in \nn}P_n(q)A_n$$ where
$P_{2n}(q)=\big[(q-x_0)^2+y_0^2\big]^n$ and $P_{2n+1}(q) =
P_{2n}(q)(q-q_0)$. Now, if $q = Q_u^{-1}v\kern1pt Q_u$ then $P_{2n}(q)
= Q_u^{-1}P_{2n}(v)Q_u$ and 
\[P_{2n+1}(q) =
Q_u^{-1}P_{2n}(v)Q_u(Q_u^{-1}v\kern1pt Q_u-q_0) =
Q_u^{-1}P_{2n}(v)(v-Q_uq_0Q_u^{-1})Q_u\] so that
$$f(q) = Q_u^{-1}\sum_{n \in \nn}\widetilde P_n(v) Q_u
A_n=Q_u^{-1}\left(\sum_{n \in \nn}\widetilde P_n(v)\widetilde
A_n\right) Q_u$$ with $\widetilde A_n = Q_uA_nQ_u^{-1}$, $\widetilde
P_{2n}(v)= \big[(v-x_0)^2+y_0^2\big]^n$ and $\widetilde P_{2n+1}(v) =
\widetilde P_{2n}(v)(v-Q_uq_0Q_u^{-1})$. Let $V \subseteq \cc^+$ be
such that $U = \{x+Iy: I \in \s, x+iy \in
V\}$. Then by the same estimates used in \cite{expansion} to prove
theorem \ref{series}, $\sum_{n \in \nn}\widetilde P_n(v)\widetilde
A_n$ converges absolutely and uniformly on compact sets in $\cc \times
V$. 

\bigbreak

Furthermore, 
\begin{align*}
[1,f(q)]
&=\bigg[Q_u,\sum_{n \in \nn}\widetilde P_n(v) Q_u
  A_n\bigg] \\ &=\bigg[Q_u,\sum_{m \in \nn}\widetilde P_{2m}(v)
  \Big(Q_u A_{2m} + (v-Q_uq_0Q_u^{-1}) Q_u A_{2m+1}\Big)\bigg]\\ &=
\bigg[Q_u,\sum_{m \in \nn}\widetilde P_{2m}(v) \Big(Q_u A_{2m} + v Q_u
  A_{2m+1} -Q_uq_0 A_{2m+1}\Big)\bigg]\\ 
  &=\bigg[1+uj,\sum_{m \in
    \nn}\widetilde P_{2m}(v)\Big((1+uj) (A_{2m}-q_0 A_{2m+1})
  + v (1+uj) A_{2m+1} \Big)\bigg]\\[3pt]
  &= \bigg[1+uj,\sum_{m \in \nn}\widetilde P_{2m}(v)\Big\{ \Big(A_{2m}-q_0 A_{2m+1}
  + vA_{2m+1} \Big) +\\[-15pt]
&\hskip200pt u \Big(jA_{2m}-jq_0 A_{2m+1}
  + vjA_{2m+1} \Big) \Big\} \bigg]
  \end{align*}
Hence, there exist holomorphic $g,h\colon V \to \cc$ such that
$[1, f(q)]$ equals 
\[\Big[1+uj, g(v)+h(v)j + u\Big(\hag(v)j-\hah(v)\Big)\Big]
=\pi\big[1,\ u,\ g(v)-u\,\hah(v),\ h(v)+u\,\hag(v)\big],\]
as required.
\end{proof}

\begin{rem}
  Let $f_1,f_2$ be regular functions on a symmetric slice domain
  $\Omega$. If $g_1,h_1$ and $g_2,h_2$ are the couples of functions
  appearing in their twistor lifts according to formula \eqref{F} then
  $f_i(v) = g_i(v) + h_i(v)j$ for all $v \in\Omega_i= \Omega \cap \cc$ and $i=1,2$. By the
  definition of regular (or star) product in
    \eqref{prodottostar}, $f_1*f_2$ splits as
\begin{equation}\label{regularproduct}
f_1*f_2= (g_1g_2-h_1 \hah_2) +(g_1h_2+h_1\hag_2)j
\end{equation}
in $L_i=\cc$. Hence, the twistor lift of $f_1*f_2$ is immediately
determined by the lifts of $f_1$ and $f_2$.
\end{rem}

If $\Omega \subset \hh$ is a symmetric slice domain that is properly
included in $\hh$, the lifting $F$ of a regular function $f: \Omega
\to \hh$ does not, in general, extend to all of $\sQ^+$.

\begin{exa}
If $f:B(0,1) \to \hh$ is defined as $f(q)=\sum_{n \in \nn}q^{2^n}$
then for $q = Q_u^{-1} v\kern1pt Q_u$ 
$$[1,f(q)] = \pi\bigg[1,\ u,\ \sum_{n \in
    \nn}v^{2^n},\ u\!\sum_{n\in\nn}v^{2^n}\bigg],$$ where $\sum_{n \in
  \nn}v^{2^n}$ cannot be holomorphically extended near any point $v$
with $|v|=1$.
\end{exa}

The quadric $\sQ$ of proposition~\ref{quadric} is of course
biholomorphic to $\cp^1\times\cp^1$; the rulings are parametrized by
$u$ and $v$. An axially-symmetric sphere $x+y\s$ can be identified
with the line
\[
\ll_v = \big\{[1,\ u,\ x+iy,\ (x+iy)u]: u\in\cc\cup\{\infty\}\big\}
\]
in $\cp^3$ defined by fixing $v=x+iy$. We may regard $\ll_v$ as a
point in the Grassmannian $\Gr_2(\cc^4)$ that parametrizes lines in
$\cp^3$, or equivalently (by means of the Pl\"ucker embedding) as a
point in the Klein quadric in $\mathbb{P}(\ext^2\cc^4)=\cp^5$.

An important feature of $F$ is that it maps each line $\ll_v$ into
another \emph{line} in $\cp^3$, since the expression in \eqref{F} is
linear in $u$ (when the right-hand side is homogenised). This is also
a consequence of the fact, shown simply in \cite{advancesrevised,
  open}, that $f$ maps $x+y\s$ to another sphere $a+\s b$ (with
$a,b\in\hh$) in an affine manner. Indeed, the projection $\pi$
establishes a bijective correspondence between lines $\cp^1$ in
$\cp^3$ and ``2-spheres'' in $\rr^4$, where a ``2-sphere'' might be an
$S^2$ of finite radius (in some 3-space), or a 2-plane, or a point. A
detailed study of this correspondence is given in \cite{shapirog}.

The action of the unit imaginary quaternion $j$ on $\hh^2$, or equivalently
\begin{equation}\label{j}
 j\colon [Z_0,Z_1,Z_2,Z_3]\mapsto 
[-\bar Z_1,\bar Z_0,-\bar Z_3,\bar Z_2]
\end{equation}
on $\cc^4$, induces a real structure
\begin{equation}\label{sigma}
\sigma\colon[\zeta_1,\zeta_2,\zeta_3,\zeta_4,\zeta_5,\zeta_6]\mapsto
[\bar\zeta_1,\bar\zeta_5,-\bar\zeta_4,-\bar\zeta_3,\bar\zeta_2,\bar\zeta_6]
\end{equation}
on $\cp^5$, expressed here relative to the basis
$\{e^{01},e^{02},e^{03},e^{12},e^{13},e^{23}\}$ of $\bigwedge^2\cc^4$
(where $e^{ij}=e^i\wedge e^j$). A fixed point of $\sigma$ corresponds
to a $j$-invariant line in $\cp^3$, i.e.\ a fibre of $\pi$ or a
``2-sphere'' of zero radius.  The fixed point set must therefore by
parametrized by $S^4$, and one can see this directly as follows.

In the above coordinates, the Klein quadric is given by
\begin{equation}\label{Klein}
\zeta_1\zeta_6-\zeta_2\zeta_5+\zeta_3\zeta_4=0,
\end{equation}
Splitting the homogeneous coordinates $\zeta_i$ into their real and
imaginary parts, the equation of the fixed point set is
\[ x_1x_6-x_2^2-x_3^2-x_4^2-x_5^2=0,\]
and defines the real quadric $\mathscr N$ of null directions in the
projectivized Lorentz space $\mathbb{P}(\rr^{1,5})$. 
The situation is summarised by the diagram
\begin{equation}\label{diag}
\begin{array}{cccc}
\Gr_2(\cc^4)  && \subset & \cp^5\\[3pt]
\cup          &&         & \cup\\[3pt]
\hp^1 &\kern-16pt =S^4=\mathscr{N} &\subset& \mathbb{P}(\rr^{1,5})
\end{array}
\end{equation}
that reflects the usual $SO(1,5)$-invariant representation of the
conformal 4-sphere.\smallbreak

Now let $f\colon\Omega \to \hh$ be a regular function, where $\Omega$
is a symmetric slice domain and
$\Omega \setminus \rr\cong\cp^1\times V$ as in \eqref{domain}, and let
$F$ be its twistor lift. We may consider the mapping
\begin{equation}\label{grass}
\tF\colon\ V\longrightarrow\Gr_2(\cc^4)
\end{equation} 
defined by $v\mapsto F(\ll_v)$. 

\begin{defin} We shall call $\tF$ the \emph{twistor transform} of $f$.
\end{defin}

The idea is simple, and exploits the philosophy of Shapiro's paper
\cite{shapirog} that twistor geometry is the complexification of Lie
sphere geometry. The domain $\hh\setminus\rr$ is a disjoint union of
the sets $x+y\s$, which are mapped by $f$ to other spheres in $\hh$,
and $\tF$ merely encodes the induced map of spheres.

\begin{thm}\label{transform}
Let $\Omega$ be a symmetric slice domain, and let $V = \Omega \cap
\cc^+$ so that $\Omega \setminus \rr\cong\cp^1\times V$. If $f:\Omega
\to \hh$ is a regular function and $\tF$ is its twistor transform
\eqref{grass}, then $\tF$ extends to $\overline{V}$ setting $\tF(\bar
v) = \sigma(\tF(v))$, and consequently it defines a holomorphic curve
on $\Omega_i=\Omega \cap \cc$. Conversely any such ``real''
holomorphic curve $G\colon V\to\Gr_2(\cc^4)$ equals the transform
$\tF$ of a unique regular map $f:\Omega \to \hh$, provided
  $\zeta_6\circ G$ is never zero.
\end{thm}

\begin{rem}\label{zeta}
The condition $\zeta_6(G(v))=0$ asserts that $G(v)$ intersects the
fibre \eqref{fibinf} over $\infty$. If the intersection is proper,
then the projection $\pi(G(v))$ is a \emph{plane} in $\hh$ rather than
a 2-sphere. But a regular function $f$ must transform $x+\s y$ into a
genuine 2-sphere for all $v=x+iy\in \cc^+$, and it is this 2-sphere
that is the projection of $\tF(v)$. When one allows $f$ to have poles,
it is possible for $\tF(v)$ to properly intersect, or to equal,
$\pi^{-1}(\infty)$. The former happens if and only if $f$ has a
2-sphere of non-removable poles, one of which has order $0$. The
latter happens if and only if $f$ has a 2-sphere of non-removable
poles, all having positive order. (For the definition and
characterization of poles, see \cite{singularities}.) For a given
``real'' holomorphic curve $G\colon V\to\Gr_2(\cc^4)$, removing those
$v=x+iy \in V$ such that $\zeta_6\circ G(v)=0$, one can reconstruct a
regular $f$ that will have poles on the 2-spheres $x+\s y$; examples
are presented after the proof.
\end{rem}

\begin{proof}[Proof of Theorem~\ref{transform}] 
It is an easy matter to determine $\tF$ explicitly when $f$ is given
by \eqref{entire} and $g,h$ by \eqref{gbhc}. It follows that 
$F(\ll_v)$ has linear equations
\begin{equation}\label{FL}
 Z_2=gZ_0-\hah Z_1,\qquad Z_3=h Z_0+\hag Z_1,
\end{equation}
whose coefficients determine the vectors 
\[ [g,-\hah,-1,0],\qquad [h,\hag,0,-1],\]
that we can wedge together. Relative to our basis $\{e^{ij}\}$,
\begin{equation}\label{-gh}
\tF(v)=[\zeta_1,\ldots,\zeta_6]=[g\hag+\hah h,\ h,\ -g,\ \hag,\ \hah,\ 1].
\end{equation}
This confirms that $\tF$ is holomorphic. The final ``1'' fixes the
projective class, and we can ignore the quadratic term because
$\tF(v)$ must satisfy \eqref{Klein}, expressing the fact that the
4-form $\tF(v)\wedge\tF(v)$ vanishes. Comparing \eqref{-gh} with
\eqref{sigma}, and recalling that $\hag(v) = \overline{g(\bar v)}$,
$\hah(v) = \overline{h(\bar v)}$ shows that $\tF$ intertwines complex
conjugation with $\sigma$.

Conversely, given a mapping $G\colon V\to \Gr_2(\cc^4)$ such that 
$G(\bar v) = \sigma(G(v))$ and with $\zeta_6 \neq0$, we may
assume that $\zeta_6\equiv1$ and recover $g=-\zeta_3$ and $h=\zeta_2$,
and (for all $v \in V$ and $u \in \cc$) we set

\begin{align*}
 \Big[1, f\big((1+uj)^{-1}v (1+uj)\big)\Big]
 &=\pi\Big[1,\ u,\ g(v)-u\,\hah(v),\ h(v)+u\,\hag(v)\Big]\\ &=
 \Big[1+uj, g(v)+h(v)j + u\Big(\hag(v)j-\hah(v) \Big) \Big]
\end{align*}
In order to verify that $f$ is regular, we compute:
\begin{align*}
&\left(\frac{\partial}{\partial x}+(1+uj)^{-1}i
  (1+uj)\frac{\partial}{\partial y} \right)f\left(x+(1+uj)^{-1}i
  (1+uj) y\right)\\ &= \left(\frac{\partial}{\partial x}+(1+uj)^{-1}i
  (1+uj)\frac{\partial}{\partial y} \right) (1+uj)^{-1}\Big\{
  g(x+iy)+h(x+iy)j\\[-10pt]
&\hskip260pt + u[\hag(x+iy)j-\hah(x+iy)]\Big\}\\ 
&= (1+uj)^{-1}\left(\frac{\partial}{\partial
    x}+i\frac{\partial}{\partial y} \right) \left\{ g(x+iy)+h(x+iy)j +
  u[\hag(x+iy)j-\hah(x+iy)]\right\},
\end{align*}
which is identically zero.
\end{proof}

\begin{exas}\label{phen}
An example of the phenomenon mentioned in Remark~\ref{zeta} is
provided by the curve
  \[ G(v) = [1,\ 0,\ i-v,\ v+i,\ 0,\ v^2+1] = \bigg[\frac1{v^2+1},\ 0,
  -\frac1{v+i},\ \frac1{v-i},\ 0,\ 1\bigg].\] This satisfies the
  reality condition, and 
  \[G(-i)=[1,0,2i,0,0,0],\qquad G(i)=[1,0,0,2i,0,0]\] are lines that
  properly intersect $\pi^{-1}(\infty)$. Retracing our steps back as
  in the previous proof, we recover $g(v)=1/(v+i)$ and $h\equiv0$, and
  \[ f(q) = Q_u^{-1}\!\Big(\frac1{v+i}+\frac u{v-i}j\Big)
  =(q^2+1)^{-1}(q-i),\] in the notation \eqref{uvu}. This is the
  so-called \emph{regular reciprocal} of $q+i$, also denoted by
  $(q+i)^{-*}$. It has a non-removable pole of order $0$ at $q=i$, and 
  poles of order $1$ at all other points of $\s$ (see \cite{open,singularities}).
  
  Let us consider a second example: a ``real'' curve $G$ giving rise
  to the regular function $f(q)=(q^2+1)^{-2}(q-i) =
  (q^2+1)^{-1}(q+i)^{-*}$, which has a sphere of non-removable poles
  of strictly positive order. It is
  \begin{align*}
 G(v) &= [1,\ 0,\ (v^2+1)(i-v),\ (v^2+1)(v+i),\ 0,\ (v^2+1)^3]\\ &=
 \bigg[\frac1{(v^2+1)^3},\ 0,-\frac1{(v^2+1)(v+i)},\ 
\frac1{(v^2+1)(v-i)},\ 0,\ 1\bigg],
  \end{align*}
  which equals $[1,0,0,0,0,0]$, i.e.\ $\pi^{-1}(\infty)$, when $v=\pm
  i$.
\end{exas}

\medbreak

If $f_r\colon\rr\to\hh$ is real analytic, we can extend
  it to a regular function $f\colon\Omega\to\hh$ for some symmetric
  slice domain $\Omega$ neighbouring $\rr$. The twistor transform of
  $f$ is none other than the \emph{complexification} of $f_r$,
  obtained using \eqref{diag}. If $f_r$ is algebraic then $\tF$ will
  be a rational curve. 

Examples~\ref{phen} fit into this scheme, and in section \ref{squartic}, 
we shall see an example in which $\tF$ is a rational quartic curve, see
  \eqref{tt}. However, an even simpler case arises when $f_r$ is the
  identity. For then
  \[ \tF(v)=\ell_v=[v^2,\ 0,\ -v,\ v,\ 0,\ 1]\] is defined for
    all $v\in\cp^1$, and the image of $\tF$ is a conic in some
    $\cp^2$. Indeed, $\ext^2(\cc^4)$ is the complexification of the
    Lie algebra
\[\mathfrak{so}(2,\hh)\cong \mathfrak{sp}(1)\oplus\mathfrak{sl}(2,\rr)
\cong \mathfrak{so}(3)\oplus\mathfrak{so}(1,2),\] and the $\cp^2$
arises by projectivizing the final summand.

To conclude, our focus on axially-symmetric spheres and the complex
structure $\jj$ has reduced the conformal group to $SO(2,\hh)$. This
provides the extra structure on the Klein quadric with which to study
regular functions.  We point out that not only do the second and third
slot in the twistor transform \eqref{-gh} determine the twistor lift,
but they also enable one to compute the transform of a star product,
thanks to formula \eqref{regularproduct}. In this way,
\eqref{prodottostar} is converted into a multiplication within the
class of holomorphic curves under consideration in
Theorem~\ref{transform}.

\begin{rem} \label{orthog} Let us consider a domain 
$V \subseteq \cc^+$ and a holomorphic curve $G\colon V\to\Gr_2(\cc^4)$
  which is not necessarily real (that is, which does not necessarily
  extend to a domain in $\cc$ which is symmetric with respect to the
  real axis). Then, provided $\zeta_6\circ G \neq 0$, the curve $G$
  will still define a mapping $F\colon \sO\to\cp^3$ on an open subset
  $\sO\cong\cp^1 \times V$ of $\sQ^+$ as in \eqref{F}, except that
  $g,\hag$ and $h,\hah$ are now unrelated. We can therefore define
  $f\colon\pi(\sO)\to\hh$ so that $f\circ\pi=\pi\circ F$. This $f$
  will be a regular function that does not necessarily extend to a
  symmetric slice domain, but that still maps $\jj$ to the orthogonal
  complex structure transforming each $\vv\in T_{f(q)} f(\pi(\sO))
  \cong \hh$ to $I_q\vv$.
\end{rem}

\section{Removing a parabola}\label{removing}

In this section, we will investigate the possibility of defining a
non-constant orthogonal complex structure on $\hh\setminus\pgamma$, where
$\pgamma$ is the parabola
$$\pgamma = \{t^2+it : t \in \rr\},$$ via the regular function $f(q) =
q^2+qi$ (which maps $\rr$ onto $\pgamma$).

\begin{thm}\label{2to1}
Let $f(q) = q^2+qi$ on $\hh$. Its degenerate set $D_f$ is empty and
its singular set $N_f$ is the $2$-plane $-i/2 + j\rr+k\rr$, whose
image is the paraboloid of revolution
\begin{equation}\label{oid}
\pGamma = \Big\{x_0+jx_2+kx_3: x_0,x_2,x_3 \in \rr,\ x_0 = 
\frs14-(x_2^2+x_3^2)\Big\}. 
\end{equation}
The restriction $f: N_f \to \pGamma$ is a bijection. Furthermore,
$f(\hh\setminus N_f) = \hh \setminus \pGamma$ and the inverse image of
every $c\in\hh\setminus \pGamma$ consists of two points. The mapping
$f\colon\hh\to\hh$ is therefore a double covering branched over
$\pGamma$ and $f(\hh\setminus \rr) = \hh \setminus \pgamma$.
\end{thm}

\begin{proof}
For any given $c \in \hh$, the polynomial $f(q)-c = q^2+qi-c$ has at
least one root by the quaternionic fundamental theorem of algebra
(see, for instance, \cite{fundamental}). Hence $f^{-1}(c)$ contains at
least a point $\alpha$. Therefore, $q^2+qi-c$ can be factored as
$(q-\alpha)*(q-\beta)$ for some $\beta \in \hh$. According to example
\ref{roots}, either
\[ f^{-1}(c)=\{\alpha,\ (\alpha -\bar \beta) \beta (\alpha -\bar
\beta)^{-1}\}\] or $\alpha$ is a point of the degenerate set $D_f$ and
$\beta = \bar \alpha$. But the latter is excluded because here
  $\alpha+\beta=-i$. Hence, $D_f=\varnothing$.

Moreover,
\[(\alpha -\bar \beta) \beta (\alpha -\bar \beta)^{-1} = 
(\alpha +\bar\alpha -i) (-\alpha-i)(\alpha +\bar \alpha -i)^{-1}.\] 
We split $\alpha = z+wj$ with $z=x_0+ix_1$ and $w=x_2+ix_3$ in
    $\cc$. Thus, $\bar\alpha = \bar z-wj$ and
\[\begin{array}{rcl}
(\alpha +\bar \alpha -i) (-\alpha-i) (\alpha +\bar \alpha -i)^{-1}
&=& (z+\bar z-i)(-z-i-wj)(z+\bar z-i)^{-1}\\[6pt] 
&=& -z-i -(z+\bar z-i)(z+\bar z+i)^{-1}wj\\[5pt]
&=&\ds -z-i +\frac{1-(z+\bar z)^2+ 2(z+\bar z)i}{1+(z+\bar z)^2}wj\\[12pt]
&=& -(z+i/2)-i/2 + e^{2\theta i}wj,
\end{array}\]
  where $\tan\theta= z+\bar z$. This corresponds to rotating $z$
  by $180^\mathrm{o}$ about the point $-i/2$, and rotating $w$ by an
  angle $2\theta$ about the origin. Thus,
\begin{equation}\label{preimage}
\textstyle f^{-1}(c) = 
\left\{z+wj,\ -(z+\frac12i)-\frac12i+e^{2i\theta}wj\right\}
\end{equation}
consists of two distinct points except when $z+\bar z = 0$ and
$z=-(z+i/2)-i/2$, which means $z = -i/2$. Hence, the singular set
$N_f$ is the $2$-plane $-i/2 +\rr j+\rr k$. 

We complete the proof observing that 
\[\textstyle f(-\frac12i+wj) = \frs14 - |w|^2 - w k,\] 
so that $f$ maps $N_f$ bijectively into the paraboloid $\pGamma$
defined by \eqref{oid}.
\end{proof}

\begin{center}
\vspace{20pt}
\scalebox{.95}{\includegraphics{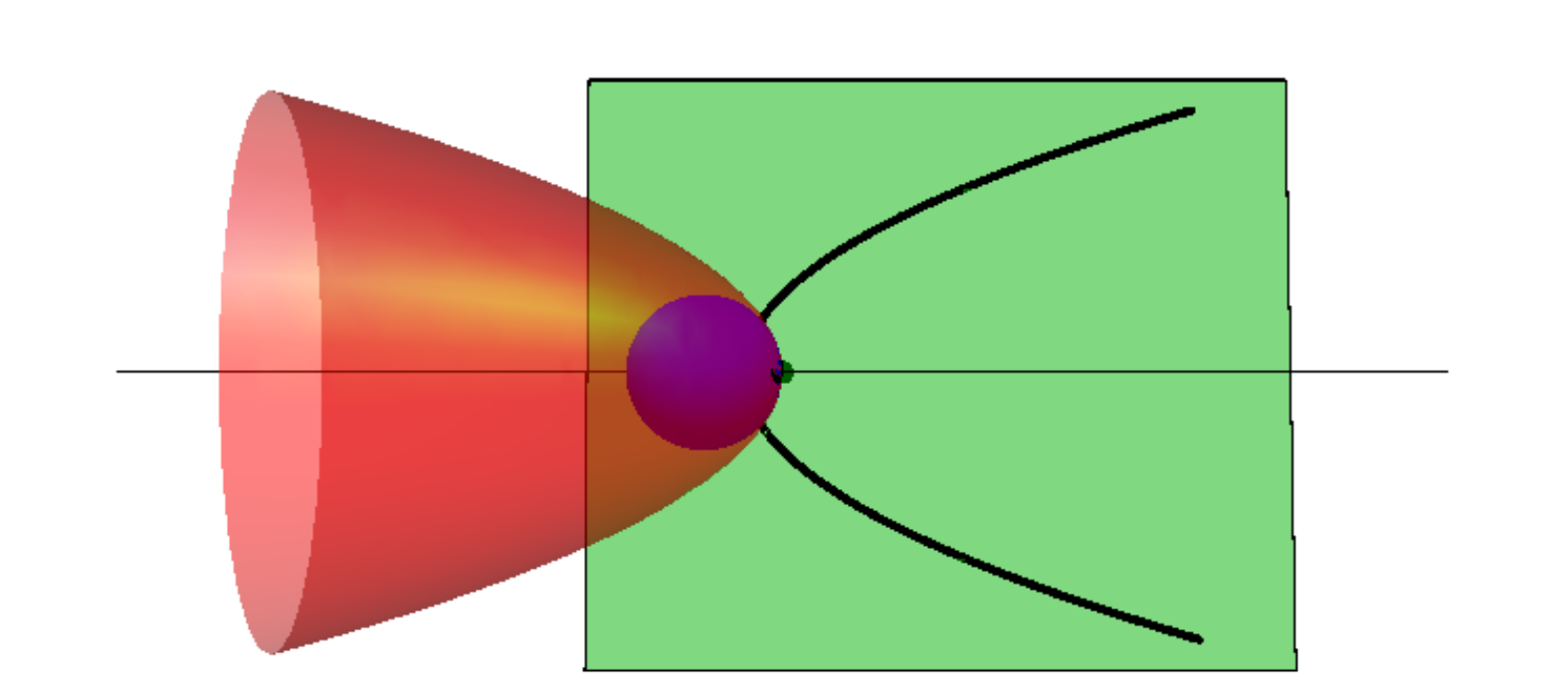}}
\vspace{0pt}
\end{center}

\parshape 1 18pt 404pt \noindent {\it Figure 1.} In $\rr^4$, the paraboloid
$\pGamma$ intersects the plane $L_i$ containing $\pgamma$ 
in the parabola's focus $(\frac14,0,0,0)$. It intersects any $3$-space
containing $\pgamma$ in a parabola with vertex at this focus. The
significance of the osculating sphere
$(x_0+\frac14)^2+x_2^2+x_3^2=\frac14$ to $\pGamma$ is explained in
section \ref{squartic}, see \eqref{osc}.

\vspace{20pt}

Inspecting the proof of the previous lemma, we find that $f$ is
injective when restricted either to the open right half-space
$$\hh^+ = \{x_0+i x_1+jx_2+kx_3 : x_0 >0\}$$ or to the open left half
space $\hh^-$. It is easily seen that their boundary $i\rr+j\rr+k\rr$
maps onto the ($3$-dimensional) solid paraboloid
$$\PGamma = \Big\{x_0+jx_2+kx_3 : x_0,x_2,x_3 \in\rr,\ 
x_0\le\frs14-(x_2^2+x_3^2)\Big\}$$ bounded by $\pGamma$ in the
3-space $\rr+j\rr+k\rr$, and that $f(\hh^+) = \hh \setminus\PGamma =
f(\hh^-)$. Note that only $\pGamma$ is of fundamental significance;
distinguishing the half spaces $\hh^+,\hh^-$ and boundary $i\rr+j\rr+k\rr$ 
determines the ``inside'' of the paraboloid, but this choice is not unique.

\bigbreak

We next use our choices to define single-valued complex structures on
a dense open set of $\hh$:

\begin{prop}\label{extends}
Let $f(q) = q^2+qi$. Let $\jj^+$ denote the complex structure induced
by the restriction $f:\hh^+ \to \hh$ on $\hh \setminus (\pgamma \cup
\PGamma)$ and let $\jj^-$ be induced by $f : \hh^- \to \hh$. Then $\hh
\setminus (\pgamma \cup \PGamma)$ is the maximal open domain of
definition for both $\jj^+$ and $\jj^-$. Indeed, both $\jj^+$ and
$\jj^-$ extend continuously to $\pGamma$ but neither of them extends
continuously to any point of $\PGamma \setminus \pGamma$.
\end{prop}

\begin{proof}
The function $f$ is injective in $\hh^+\cup N_f$. Its inverse
$$f^{-1}:(\hh\setminus \PGamma) \cup \pGamma\to \hh$$ 
is continuous by the open mapping theorem, since the degenerate set of
$f$ is empty. Hence, $p \mapsto I_{f^{-1}(p)}$ is continuous and the
structure $\jj^+$ can be extended continuously to $\pGamma$ setting
$\jj^+_{p}\vv=I_{f^{-1}(p)}\vv$ for all $p\in \pGamma$.

Now, let $p \in \PGamma\setminus \pGamma$: since $f^{-1}$ is defined and
continuous on $(\hh\setminus \PGamma) \cup \pGamma$, the structure 
$\jj^+$ admits a continuous extension to $p$ if and only if for $a \in \hh 
\setminus (\PGamma\cup \pgamma)$
$$\lim_{a \to p} I_{f^{-1}(a)}$$ exists. We saw in the proof of the
previous lemma that $f^{-1}(p) = \{q, q'\}$ with $q= ti+wj$ and
$q'=-(t+1)i+wj$ for some $t >-1/2$ and $w\in\cc$. Since $q$ and
$q'$ are both in $\overline{\hh}^+$, the aforementioned limit exists
if and only if $I_q = I_{q'}$. However, by direct computation
$$I_{q} = \frac{ti+wj}{|ti+wj|}\ne\frac{-(t+1)i+wj}{|-(t+1)i+wj|} 
= I_{q'},$$ and the proof is complete.\end{proof}

Finally, let us compare the two structures that we introduced.

\begin{thm}\label{JJJJ}
The structures $\jj^+$ and $\jj^-$ have the following properties on
$\hh \setminus (\pgamma \cup \PGamma)$:
\begin{enumerate}
\item The four structures $\jj^+,\jj^-,-\jj^+,-\jj^-$ are distinct
  outside of $L_i=\cc$.
\item At every point of $\{x_0+ix_1: x_0 > x_1^2\}$, the
  structures $\jj^+$ and $\jj^-$ both coincide with the left
  multiplication by $-i$.
\item At every point of $\{x_0+ix_1: x_0 < x_1^2\}$, the structures 
  $\jj^+$ and $-\jj^-$ coincide; in $\{x_0+ix_1:\quad x_0 <
  x_1^2,\ x_1>0\}$, they both coincide with left multiplication by
  $i$; in $\{x_0+ix_1: x_0 < x_1^2,\ x_1<0\}$, they coincide with the
  left multiplication by $-i$.
\end{enumerate}
At every point of the paraboloid $\pGamma$, the extended structures
$\jj^+$ and $\jj^-$ coincide.
\end{thm}

\begin{proof}
For every point $c\in 	\cc\setminus (\pgamma \cup \PGamma)$,
\[\textstyle
f^{-1}(c) = \left\{z,\ -(z+\frac12i)-\frac12i\right\},\]
for some $z\in\cc$ with $\Re(z)>0$. If $c \in \{x_0+ix_1: x_0 >
x_1^2\}$ then $z$ and $-(z+i/2)-i/2$ both belong to $\{x_0+ix_1:
-1<x_1<0\}$ so that $\jj^+_c$ and $\jj^-_c$ both coincide with the
left multiplication by $-i$. If $c \in \{x_0+ix_1: x_0 <
x_1^2,\ x_1>0\}$ then $z \in \{x_0+ix_1 : x_1>0\}$ and $-(z+i/2)-i/2$
belongs to $\{x_0+ix_1: x_1<-1\}$ so that $\jj^+_c$ coincides with
the left multiplication by $i$ and $\jj^-_c$ coincides with the left
multiplication by $-i$. Moreover, if $c \in \{x_0+ix_1: x_0 <
x_1^2,\ x_1<0\}$ then $z \in \{x_0+ix_1: x_1<-1\}$ and $-(z+i/2)-i/2
\in \{x_0+ix_1: x_1>0\}$ so that $\jj^+_c$ coincides with the left
multiplication by $-i$ and $\jj^-_c$ coincides with the left
multiplication by $i$.

Every point $c \in \hh \setminus (\pgamma \cup \PGamma \cup \cc)$ has
preimage \eqref{preimage} for some $z \in\cc$ with $\Re(z) >0$ and
non-zero $w\in \cc$; by direct inspection, $w\ne\pm
e^{2\theta i}w$ so that $\jj^+_c\ne\pm\jj^-_c$.

The final assertion follows from the fact that the preimage of each
$c\in \pGamma$ consists of one point.
\end{proof}

\section{A rational quartic scroll}\label{squartic}

We apply the theory of section~\ref{twist} to the example of
section~\ref{removing}, again with $f(q)=q^2+qi$.

From the point of view of complex structures, $f$ maps the plane
$L_i$ containing the parabola $\pgamma=\{t^2+ti:t\in\rr\}$
onto itself. It maps $-i/2$ to the focus $\frac14$, but is $2:1$
elsewhere in $L_i$. It follows that, pointwise for $q\in L_i$, the
induced complex structure $\jj^f_{f(q)}$ is determined by the action
of $f$ on $L_i$. This will be a key feature of the following
interpretation, in which the multi-valued structure $\jj^f$ is encoded
in the twistor lift of $f$. We shall see that the action of $f$ on
$L_i$ gives rise to double points in the ``graph'' of $\jj^f$.

The twistor lift is given by \eqref{F}, where
\begin{equation}\label{+gh}
g(v)=v^2+iv,\quad \hag(v)=v^2-iv,\quad h(v)=0.
\end{equation}
 This extends to a holomorphic mapping $F:\sQ\to\cp^3$ by allowing $v$
 to assume values in $\cc$ rather than just $\cc^+$. To be more
 precise, we shall adopt homogeneous coordinates for $\cp^3$ throughout
 this final section. Then
\begin{equation}\label{Fhom}
 F[st,su,tv,uv] = \big[s^2t,\ s^2u,\ t(v^2+isv),\ u(v^2-ivs)\big],
\end{equation}
with $([s,v],[t,u])\in\cp^1\times\cp^1$, is well defined and
consistent with \eqref{F}.

If we set $Z_0=s^2t$, $Z_1=s^2u$ etc., then
\[ Z_1Z_2-Z_0Z_3= 2is^3tuv,\qquad Z_1Z_2+Z_0Z_3=2s^2tuv^2.\]
It follows that the image of $F$ lies in the quartic surface $\sK$
defined by the equation $K=0$, where
\begin{equation}\label{K}
K(Z_0,Z_1,Z_2,Z_3) = (Z_1Z_2-Z_0Z_3)^2 +2Z_1Z_0(Z_1Z_2+Z_0Z_3).
\end{equation}
Observe that $\sK$ is a \emph{real} subvariety of the twistor space,
invariant by the antilinear involution \eqref{j}.

We already know from Theorem~\ref{2to1} that $f$ is generically
$2:1$. If $f(q_1)=f(q_2)$ then the two pairs of points in the quadric
$\sQ$ of Proposition~\ref{quadric} lying over $q_1,q_2$ must map to
the typically four points of $\sK$ lying over the common image
point. This shows that $F$ is generically $1:1$. In order to describe
the action of $F$ more explicitly, consider the intersection of $\sK$
with the quadric $Z_1Z_2-Z_0Z_3=0$, which we can in fact identify with
$\sQ$. This intersection also satisfies $Z_0Z_1Z_2=0$. With the
  notation
\begin{equation}\label{mij}
m_{ij}=\{[Z_0,Z_1,Z_2,Z_3]: Z_i=0=Z_j\},
\end{equation}
it must therefore consist of the four lines
$m_{01},m_{02},m_{13},m_{23}$ forming a ``square''.

\begin{thm}\label{bir}
The mapping $F$ is a birational equivalence between $\sQ$ and
$\sK$. In fact, $F$ maps $\sQ$ onto $\sK$ and is injective on the
complement of $\mm_{02}\cup\mm_{13}$ over which it is a $2:1$ branched
covering. The singular locus of $\sK$ is the union
$\mm_{02}\cup\mm_{13}\cup\mm_{01}$, and the set of cusp points is
$\mm_{01}\cup\left\{\left[0,1,0,\frac14\right],\left[1,0,\frac14,0\right]\right\}$.
\end{thm}

\begin{proof}
To prove that $F$ is onto, first suppose that $Z_0\ne0$ and
$Z_1\ne0$. Then we can define \[ u/t=Z_1/Z_0,\qquad
2iv/s=(Z_1Z_2-Z_0Z_3)/(Z_0Z_1),\] and $[Z_i]$ has a unique inverse
image.

If $Z_0=0$ then $Z_1Z_2=0$, whereas if $Z_1=0$ then $Z_0Z_3=0$. These
equations define the lines $m_{02},m_{13},m_{01}$, so we need to
  consider the three possibilities in turn:
\begin{itemize}
\item If $Z_0=Z_2=0$ then $t=0$ and $F$ maps $[s,v]\in\cp^1$ to $[0,
  s^2, 0, v^2-isv]\in\mm_{02}$.
\item If $Z_1=Z_3=0$ then $u=0$ and $F$
maps $[s,v]\in\cp^1$ to $[s^2,0,v^2+isv,0]\in\mm_{13}$.
\item If $Z_0=0=Z_1$ then $s=0$ and $F$ maps $[t,u]\in\cp^1$ to
  $[0,0,t,u]$, so is the identity on $\mm_{01}$.
\end{itemize}
It follows that $F$ is injective unless $[t,u]$ equals $[1,0]$ or
$[0,1]$ \emph{and} $[s,v]$ does not equal $[1,0]$ or $[0,1]$.
Moreover, restricted to $\mm_{02}\cup\mm_{13}$, the map $F$ is $2:1$
except over the vertices
\begin{equation}\label{pinch}
 [1,0,0,0],\quad[0,1,0,0],\quad[0,0,1,0],\quad[0,0,0,1]
\end{equation}
of the square, where it acts as the identity.

\smallbreak

Any singular point $[Z_i]$ of $\sK$ must satisfy
\[\begin{array}{l}
\ds 0= \frs{\partial K}{\partial Z_0}=-2Z_3(Z_1Z_2-Z_0Z_3)+ 2Z_1(Z_1Z_2+2Z_0Z_3)\\[6pt]
\ds 0= \frs{\partial K}{\partial Z_1}=2Z_2(Z_1Z_2-Z_0Z_3)+2Z_0(2Z_1Z_2+Z_0Z_3)\\[6pt]
\ds 0= \frs{\partial K}{\partial Z_2}=2Z_1(Z_1Z_2-Z_0Z_3)+2Z_0Z_1^2\\[6pt]
\ds 0= \frs{\partial K}{\partial Z_3}=-2Z_0(Z_1Z_2-Z_0Z_3)+2Z_0^2Z_1.
\end{array}\]
The last two equations imply that $Z_0=0$ or $Z_1=0$. Since $[Z_i]$
satisfies \eqref{K}, we know that it must lie in
$\mm_{02}\cup\mm_{13}\cup\mm_{01}$. Conversely, any point in the union
of the three lines satisfies all four equations above, and is
singular.

To find the cusp points, one can compute the $2\times2$ minors
(determinants) of the Hessian
$\Big(\!\hbox{\large$\frac{\partial^2K}{\partial Z_i\partial
    Z_j}$}\!\Big)$ at the singular points. When $Z_0=Z_1=0$, all
minors are identically zero, and every point of $\mm_{01}$ is a
cusp. Over $\mm_{02}$ the relevant determinant equals
$4Z_1^3(4Z_3-Z_1)$, and over $m_{13}$ it is $4Z_0^3(4Z_2-Z_0)$,
indicating additional cusps at $[0,1,0,\frac14]$ and
$[1,0,\frac14,0]$.\end{proof}

Recall from \eqref{fibinf} that $m_{01}=\pi^{-1}(\infty)$. To make the
twistor projection $\pi$ more explicit, set
\begin{equation}\label{q}
 \tq=x_0+ix_1+(x_2+ix_3)j=w_1+w_2j,\qquad w_1,w_2\in\cc,
\end{equation}
so that $[Z_i]\in\pi^{-1}(\tq)$ if and only if
$Z_2+Z_3j=(Z_0+Z_1j)(w_1+w_2j)$. (In practice, we will want to
  consider points $\tq=f(q)$, so the tilde is to avoid confusion).
This implies that $\pi^{-1}(\tq)$ has equations
\begin{equation}\label{Zw}
 Z_2 = w_1Z_0-\ol w_2Z_1,\qquad
 Z_3 = w_2Z_0+\ol w_1Z_1,
\end{equation}
and coordinates
\begin{equation*}
[\zeta_1,\zeta_2,\zeta_3,\zeta_4,\zeta_5,\zeta_6] 
= \left[|w_1|^2+|w_2|^2,w_2,-w_1,\bar w_1, \bar w_2,1\right]
\end{equation*}
in $\cp^5$ (cf.~\eqref{FL} and \eqref{-gh}), so that it is fixed by
the action of $\sigma$.  We obtain
\begin{itemize}
\item $Z_2=0=Z_3 \Ra w_1=0=w_2 \Ra\tq=0$,
\item $Z_0=0=Z_2$ or $Z_1=0=Z_3$\Ra $w_2=0\Ra\tq\in L_i$,
\item $Z_0=0=Z_3$ or $Z_1=0=Z_2$\Ra $w_1=0\Ra\tq\in L_i^\perp$.
\end{itemize}

\noindent Therefore $m_{23}=\pi^{-1}(0)$ and $m_{02},m_{13}$ both
project to $L_i\cup\infty$. It follows from Theorem~\ref{bir} that
topologically, {\it $\sK$ is obtained from $\sQ$ by carrying out a
  $\zz_2$ identification on each of two punctured planes lying over
  $L_i\setminus\{0\}$.} The main features of $\sK$ are represented by
Figure~2.

\begin{center}
\vspace{-10pt}
\scalebox{.65}{\includegraphics{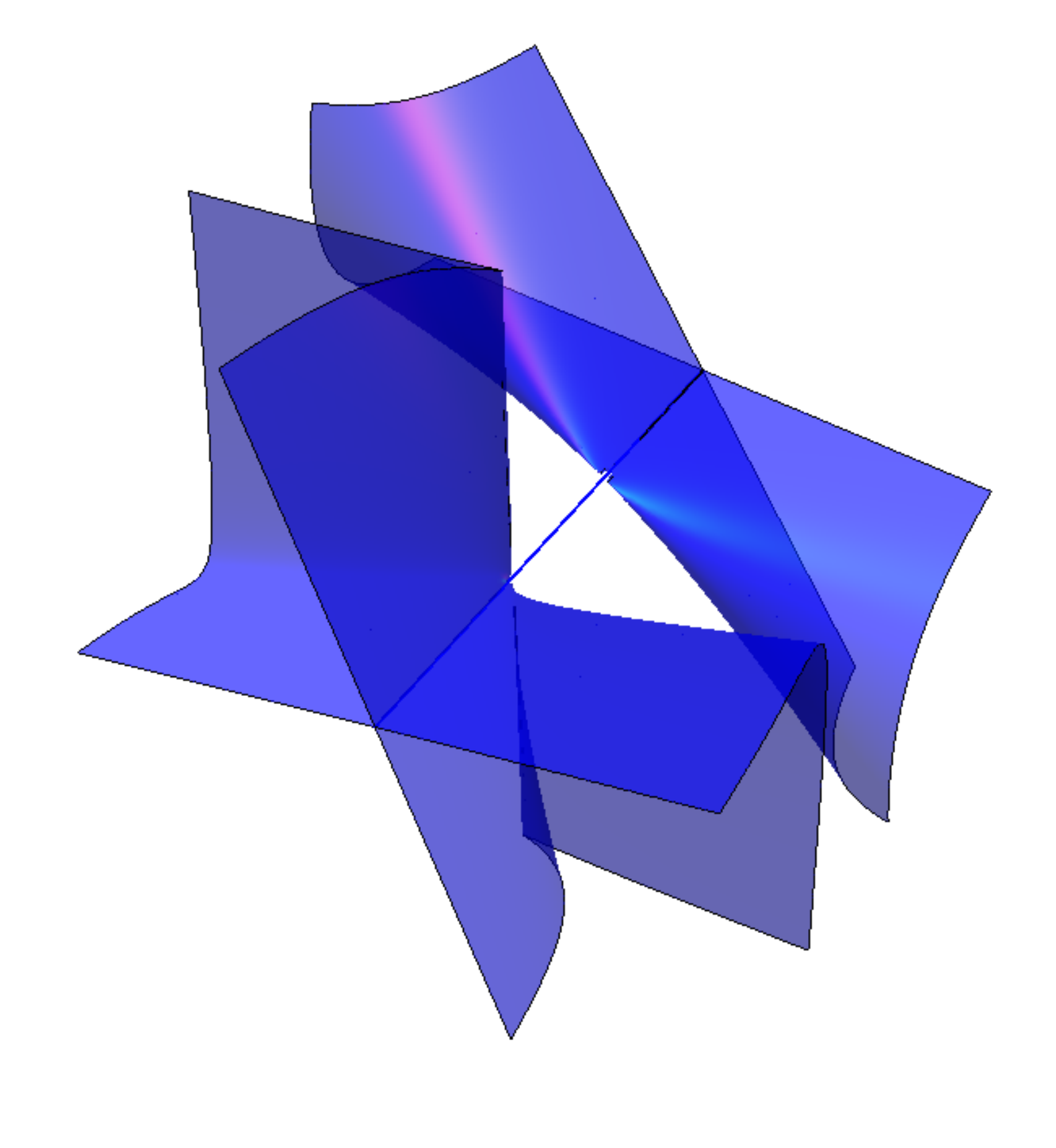}}
\vspace{-20pt}
\end{center}

\parshape 1 18pt 404pt \noindent {\it Figure 2.} The real affine slice
of $\sK$ defined by $K(x,y,z,\frac14)=0$. The vertical line is
$m_{01}$, and the diagonal one containing an extra isolated cusp
corresponds to $m_{02}$.

\vspace{20pt}

From \eqref{-gh} and \eqref{+gh}, the twistor transform of $f$
is the mapping $\tF:\cp^1\to\cp^5$ given by
\begin{equation}\label{tt}
\tF(v)=[v^4+v^2,\ 0,\ -v^2-iv,\ v^2-iv,\ 0,\  1].
\end{equation}
Its image lies in a 2-quadric, the intersection of the Klein quadric
with two of its tangent hyperplanes, those at the points
$[0,1,0,0,0,0]$ and $[0,0,0,0,1,0]$ of $\cp^5$ that correspond to
$\mm_{02}$ and $\mm_{13}$ respectively. It has a cusp at
$[1,0,0,0,0,0]$ that corresponds to $\mm_{01}$. By construction, $\sK$
is the ruled surface or \emph{scroll} containing, for each $v\in\cc$,
the line $F(\ll_v)=\tF(v)$, intersection of the planes
\begin{equation}\label{planes}
Z_2 = (v^2+iv)Z_0\quad\hbox{and}\quad Z_3 =(v^2-iv)Z_1.
\end{equation} This line is a twistor fibre if and only if it satisfies
the reality condition $\sigma(\tF(v))=\tF(v)$, equivalently, $v=\bar
v=x\in\rr$. Indeed,
\[ F(\ll_x) =
\{[Z_0,\,Z_1,\,(x^2+ix)Z_0,\,(x^2-ix)Z_1]:[Z_0,Z_1]\in\cp^1\} =
\pi^{-1}(x^2+ix),\] for each $x\in\rr$. Of the six lines
  \eqref{mij}, only $\mm_{01}=\ell_\infty$ and $m_{23}=\ell_0$ belong
  to the ruling parametrized by the twistor transform $\tF$.

\begin{rem} It follows from the above descriptions that the
  rational ruled quartic surface $\sK$ is of type 1(iii) in
  Dolgachev's list \cite[Theorem 10.4.15]{dolg} and, as explained in
  that book, of type IV(B) in Edge's list \cite{edge}, these
  classifications building on work of Cayley and Cremona. In classical
  terminology, the singular locus of $\sK$ is the \emph{double curve},
  and the singular points \eqref{pinch} are the \emph{pinch points} of
  $\sK$ \cite{dolg}. The singular locus is effectively determined by
  the properties of $\tF$ mentioned immediately after \eqref{tt}, and
  the equation \eqref{K} is a variant of Rohn's normal form, see
  \cite[Lemma~1.10 and \S3.2.8]{polo}.
\end{rem}

Substituting $Z_2=w_1Z_0$ and $Z_3=\ol w_1Z_1$ in \eqref{K}, we see
that points in $\pi^{-1}(L_i)\cap\sK$ are characterised by the
equation
\[Z_0^2Z_1^2(x_1^2-x_0)=0,\] confirming that each point of $L_i$ is
covered by exactly two double points of $\sK$. 
These affine singular points arise naturally as intersection points of
lines in the ruling. Indeed, $F(\ll_v),F(\ll_w)$ meet where
\[\begin{array}{l}
Z_2 = (v^2+iv)Z_0 = (w^2+iw)Z_0,\\[3pt]
Z_3 = (v^2-iv)Z_1 =(w^2-iw)Z_1,
\end{array}\]
 yielding $Z_0=0=Z_2$ and $v+w=i$, or $Z_1=0=Z_3$ and $v+w=-i$. 

The lines $F(\ell_v)$, $F(\ell_{\pm i-v})$ are distinct unless $v=\pm
i/2$. In the former case, they meet $m_{02}$ or $m_{13}$ (depending on
the sign) transversely at a node of $\sK$. Moreover, if they meet over
$\pgamma$ one of the two lines is a twistor fibre. On the other hand,
each of the two double lines $\ell_{i/2},\ell_{-i/2}$ is tangent to
$\sK$ at a cusp lying over the focus of $\pgamma$, i.e.,\ at a point of
$\sK\cap\pi^{-1}(\frac14)$. One of these two pinch points is
  shown in Figure~2, and we can establish a direct link with Figure~1
  as follows. If we use \eqref{Fhom}, we obtain
  \begin{equation*}\textstyle
 F\big[1,u,\frac i2,\frac i2u\big] = \big[1,u,-\frac34,\frac14u\big],
\end{equation*}
 for $u\in \mathbb{C}$. We therefore see that 
 \[\textstyle
F(\ell_{i/2}) = \Big\{\big[1,u,-\frac34,\frac14u\big] :
u\in\cc \cup \{\infty\}\Big\}\] and that the projection
$\pi(F(\ell_{i/2}))=\pi(F(\ell_{-i/2}))$ is parametrized
stereographically in $\hh=\rr^4$ by
\begin{equation}\label{osc}
\frac1{4(1+|u|^2)}\Big(|u|^2-3+4uj\Big) = 
\frac1{4(1+|u|^2)}\Big(|u|^2-3,\ 0,\ 4\,\Re\, u,\ 4\,\Im\, u\Big). 
\end{equation}
(here $\Im$ denotes the complex imaginary part, $\Im : \cc \to \rr$).
This is the sphere $$\textstyle
f\left(\frac12\s\right)= -\frac 14 + \frac i2 \s
\ =\ \left\{x_0+x_2j+x_3k:
(x_0+\frac14)^2+x_2^2+x_3^2=\frac14\right\}$$ visible in Figure~1,
containing the focus of $\gamma$.

\medbreak

The next result asserts that the surface $\sK$ is in effect determined
by the parabola $\pgamma = f(\rr)$.

\begin{prop}\label{Nullst}
Any algebraic surface in $\cp^3$ containing $\pi^{-1}(\pgamma)$
contains $\sK.$
\end{prop}

\begin{proof} 
 Let $p=p(Z_0,Z_1,Z_2,Z_3)$ be a polynomial that vanishes on
$F(\ll_x)$ for all $x\in\rr$.  We shall actually show that $K$ divides
$p$.

Referring to \eqref{Fhom}, we have that $P=p\circ F$ vanishes on all
$\ll_x$, so
  \[ P(Z_0,Z_1,xZ_0,xZ_1)=0,\qquad \forall\, [Z_0,Z_1]\in\cp^1,\ \
  x\in\rr.\] Since $P$ is a polynomial, we can replace $x\in\rr$ by
  $x\in\cc$, and so $P$ vanishes on the quadric $\sQ$. It follows from
  Theorem~\ref{bir} that $p$ vanishes on $\sK$.

The polynomial $K$ must be irreducible because no line or quadric in
$\cp^3$ could contain $\pi^{-1}(\pgamma)$. It follows from the
Nullstellensatz that $K$ divides $p$.
\end{proof}

In view of the proposition, properties of $\sK$ are naturally
associated to the 4-dimensional conformal geometry of the
parabola. We are now in a position to consolidate the analysis
  that we are carrying out in parallel to section~\ref{removing}.

\begin{thm} \label{disc}
Let $\tq=f(q)\in\hh$. The cardinality of the fibre
$\pi^{-1}(\tq)\cap\sK$ is different from $4$ in the following cases:
\begin{enumerate}
\item $\tq\in \pgamma$ iff $\pi^{-1}(\tq)\subset\sK$;
\item $\tq\in L_i\setminus\pgamma$ iff $\pi^{-1}(\tq)$ contains
  exactly two singular points of $\sK$;
\item $\tq\in\pGamma\setminus\left\{\frac14\right\}$ iff
  $\pi^{-1}(\tq)$ is tangent to $\sK$ at two smooth points.
\end{enumerate}
\end{thm}

\begin{proof}
We tackle this by investigating the intersection of a given fibre
$\pi^{-1}(\tq)$ with the lines $F(\ll_v)$ as $v$ varies. Referring to
\eqref{Zw} and \eqref{planes}, we need to solve the equations \[
(v^2+iv)Z_0 = w_1Z_0-\ol w_2Z_1,\qquad
(v^2-iv)Z_1 = w_2Z_0+\ol w_1Z_1\]
which yield
\[ \Big((v^2-iv-\ol w_1)(v^2+iv-w_1)+|w_2|^2\Big)Z_0Z_1=0.\]
If $Z_0=0$ or $Z_1=0$ (two antipodal points on the fibre) then
$w_2=0$ and
\begin{equation}\label{or}
\ol w_1=v^2-iv\quad\hbox{ or }\quad w_1=v^2+iv.
\end{equation}
For fixed $\tq=w_1\in L_i$, there is a total of four solutions
$\ell_v$ unless $\tq=\frac14$ is the focus of $\pgamma$. We already
know that these solutions are associated to two distinct singular
points of $\sK$.

The line $F(\ell_v)$ can only equal a fibre $\pi^{-1}(\tq)$ if hits both
points $Z_0=0$ and $Z_1=0$ of that fibre. This forces the ``or'' in
\eqref{or} to be an ``and'', whence $v=x\in\rr$ and $\tq\in\pgamma$.
This is statement 1, and statement 2 also follows.

If $Z_0Z_1\ne0$ then it is the vanishing of
\[ R(v)=v^4+(1-2x_0)v^2-2x_1v + C,\]
where $C=|\tq|^2=\sum_{i=0}^3 x_i^2$, that determines the points of
$\pi^{-1}(\tq)\cap\sK$. Since the coefficients are real, the roots of
$R(v)$ must occur in conjugate pairs.

A computation of the discriminant of $R(v)$ shows that it equals
$16$ times
\[\begin{array}{l}
D\ =\ C-8C^2+16C^3-8Cx_0+32C^2x_0+24Cx_0^2-32C^2x_0^2-32Cx_0^3\\[5pt]
\hskip30pt
+16Cx_0^4-x_1^2+36Cx_1^2+6x_0x_1^2-72Cx_0x_1^2-12x_0^2x_1^2+8x_0^3x_1^2-27x_1^4
\end{array}\]
When $x_1=0$, this expression reduces to $C(-1 + 4C + 4x_0 - 4x_0^2)^2$,
so the zero set of $D$ contains
\begin{equation}\label{critical}
  x_0 = \frs14 - x_2^2-x_3^2,\qquad x_1=0,
\end{equation}
which is the equation of $\Gamma$. We need to show that there are no
other zeros.

When we substitute $x_i=\pm t$, the factor $D$ becomes a polynomial of
degree 6 in $t$ with leading term $9216t^6$ irrespective of the choice
of signs. Hence it is positive on the boundary of a sufficiently large
cube, and must attain a minimum inside the cube. It suffices to show
that the only critical points of $D$ occur at \eqref{critical}. 
It turns out that $\partial D/\partial x_i$ is divisible by $x_i$ for
$i=1,2,3$, so we set
\[\delta_i=\left\{\begin{array}{ll}
\ds\frac{\partial D}{\partial x_0}\quad & i=0,\\[20pt]
\ds\frac1{x_i}\frac{\partial D}{\partial x_i} & i=1,2,3.
\end{array}\right.\]
A computation shows that the system $\delta_i=0$ for $i=0,1,2,3$ only
has real solutions with $x_1=0$, and there are no further solutions
with $x_2=0$ or $x_3=0$.

We already know that the quartic has no singular points over
  $\hh\setminus L_i$. If two of the four points coincide at a
  smooth point of $\sK$ then $\sK$ is tangent to the fibre at this
  point. Statement 3 follows.
\end{proof}

Theorem~\ref{disc} provides essentially the same information as
Theorem~\ref{JJJJ}. In a neighbourhood $U$ of a generic point of
$\hh,$ the quartic $\sK$ is a $4:1$ covering of $U$ whose leaves are
holomorphic sections of the twistor space $\cp^3$. In particular,
$\sK$ contains the graphs of the complex structures
$\pm\jj^+,\pm\jj^-$ defined by Proposition~\ref{extends} on the
complement $\hh\setminus(\pgamma\cup\PGamma)$ of the parabola
$\pgamma$ and the solid paraboloid $\PGamma$. These two pairs coincide
over $L_i$, but the two leaves cross transversely (so both $\jj^+$ and
$\jj^-$ are well defined) over $L_i\setminus \left(\gamma \cup
  \left\{\frac14 \right\}\right)$. On the other hand, along directions
approaching $\pGamma$ the graphs become vertical, so the covariant
derivatives $\nabla\jj^\pm$ (computed in any conformally flat metric)
are unbounded.

 The non-trivial computations required in the proof of statement 3
 highlight the difficulty in computing the discriminant locus of some
 quite simple algebraic surfaces relative to the twistor projection
 $\pi$. In this paper, the quartic $\sK$ was ``reverse engineered''
 via Proposition~\ref{Nullst}. The problem of understanding the
 discriminant locus of generic quartics (or even cubics) remains
 open. The last proof also serves to demonstrate the relative
 effectiveness of the approach in section~6, which achieved the same
 result using quaternions.\smallbreak

Having identified the surface $\sK$, we can proceed with the proof of
the next result by mimicking that of \cite[Theorem~3.10]{viaclovsky}.

\bigbreak

\begin{thm}\label{nonex}\mbox{}
\vspace{-5pt}
\begin{enumerate}
\item There is no orthogonal complex structure whose maximal domain of
  definition is $\hh\setminus\pgamma$.
\item Let $J$ be an OCS defined on
  $\hh\setminus(\pgamma\cup\PGamma)$. If, for each $p\in\PGamma$, the
  limit points of $J$ at $p$ are among the limit points of
  $\jj^+,\jj^-,-\jj^+,-\jj^-$ at $p$, then $J$ is itself one of these
  four complex structures.
\end{enumerate}
\end{thm}

\begin{proof} 
  1. First suppose that $J$ is an OCS defined on
  $\hh\setminus\pgamma$. Let $\sS$ be the graph of $J$, which is a
  complex submanifold of $\pi^{-1}(\hh\setminus\pgamma)$ by
  \cite{Salamon1985,deBartolomeisNannicini}. If the holomorphic
  function $K$ vanishes on an open subset of $\sS$ then it vanishes
  identically, and $\sS\subset\sK$. This is not possible because $\sK$
  is not transversal to the twistor fibres over $\pGamma$, while $J$
  is smooth over $\pGamma$. It follows that $\sS\cap\sK$ has complex
  dimension at most $1$ at each point.

Now consider the set \begin{equation}\label{A}
 A=\sS\cap(\cp^3\setminus\sK)=\sS\setminus(\sS\cap\sK).
\end{equation}
Since $\sS$ is a complex analytic subset of
$\pi^{-1}(\hh\setminus\gamma)$, $A$ is a complex analytic subset of
the smaller open subset $\cp^3\setminus\sK$. We can therefore apply
Bishop's lemma \cite[Lemma~9]{Bishop} in an attempt to extend $A$
across the algebraic variety $B=\sK$ of complex dimension $k$ with
$k=2$. In order to do this we need to know that the $2k$-dimensional
Hausdorff measure of $\overline A\cap\sK$ is zero, where $\overline A$
is the closure of $A$ in $\cp^3$. But $\overline A\cap\sK$ is a
  subset of \[(\sS\cap\sK)\cup\pi^{-1}(\gamma)\cup\pi^{-1}(\infty),\]
  which is a union of real submanifolds of dimension less than and
  equal to $3$.

  The extension is therefore possible, and we conclude that $\overline
  A$ is complex analytic. Moreover, $\overline A =\overline\sS$.
  Indeed, for all $p \in \overline\sS$ and for all neighbourhoods $U$
  of $p$ in $\cp^3$, the fact that $U \cap \sS \ne\varnothing$ implies
  that $U \cap A = (U \cap \sS) \setminus (\sS \cap
  \sK)\ne\varnothing$ because $\sS\cap\sK$ has complex dimension at
  most $1$. Since $\overline \sS$ is complex analytic, it follows
  from standard results (such as Mumford's treatment \cite{Mumford})
  that $\overline\sS$ is algebraic. However, it then has degree one
  since it meets a generic line in just one point, so $J$ must be
  conformally constant and extend at least to $\hh$ minus a
  point.\smallbreak

  2. Now suppose that $J$ is an OCS defined on
  $O=\hh\setminus(\gamma\cup\PGamma)$, with the stated limiting
  property over $\PGamma$. The graph $\sS$ of $J$ is a complex
  submanifold of $\pi^{-1}(O)$. If $\sS\cap\sK$ has complex dimension
  $2$ at some point then $\sS\subset\sK$, and $\sS$ must coincide with
  one of the smooth branches of $\sK$ over $O$, namely the graph of
  one of $\jj^+,\jj^-,-\jj^+,-\jj^-$. We can therefore assume that
  $\sS\cap\sK$ has complex dimension at most $1$ at each point.
  
Define $A$ as in \eqref{A}. Then $\overline A\cap\sK$,
a subset of
\[ (\pi^{-1}(\PGamma)\cap\sK)\cup(\sS\cap\sK)\cup\pi^{-1}(\gamma)
\cup\pi^{-1}(\infty),\] again has $4$-dimensional Hausdorff measure
equal to zero. Indeed, $\pi^{-1}(\PGamma)\cap\sK$ has real dimension
$3$ in view of Theorem \ref{disc}. The fact that $A$ is an analytic
subset of $\cp^3\setminus\sK$ follows because $\sS$ is analytic in
$$\cp^3 \setminus (\pi^{-1}(\PGamma) \cup \pi^{-1}(\gamma) \cup\pi^{-1}(\infty))$$
and $A = \sS \setminus \sK$ intersects neither $\pi^{-1}(\gamma)
\cup\pi^{-1}(\infty)$ (as before) nor $\pi^{-1}(\PGamma)$ (thanks to
the hypothesis on the limit points). Applying Bishop's lemma, we
deduce that $\overline A=\overline\sS$ is analytic. Arguing as above,
$J$ must be conformally constant, but such a $J$ cannot coincide with
$\pm\jj^\pm$ over $\PGamma$.
\end{proof}

To sum up, Proposition~\ref{Nullst} and Theorem~\ref{disc} imply
that the paraboloid $\pGamma$ is determined algebraically by the
parabola $\pgamma$. The vertex of $\pGamma$ is the focus of $\pgamma$,
and the two points of $\sK$ lying above this vertex are cusps.
  On the other hand, $\pi^{-1}(\pGamma)$ is not a complex submanifold
  of $\sK$ because the 2-plane $N_f=f^{-1}(\pGamma)$ is not
  $\jj$-holomorphic, see \eqref{oid}.

The subsets $\pgamma$, $L_i$ and $\pGamma$ of $\hh$ (together with the
point at infinity) constitute the discriminant locus of $\sK$, and no
OCS whose graph lies in $\sK$ can be smoothly extended over $\gamma$
and $\pGamma$. Theorem~\ref{nonex} tells us that such an OCS is
characterized analytically by its limiting values over a 3-dimensional
region, such as $\PGamma$, with boundary $\pGamma$.

\bigskip
\footnotesize
\noindent\textit{Acknowledgments.} Graziano Gentili and Caterina Stoppato 
are partially supported by GNSAGA of INdAM and by the projects FIRB 2008 
``Geometria Differenziale Complessa e Dinamica Olomorfa'' and PRIN 2010-2011 
``Variet\`a reali e complesse: geometria, topologia e analisi armonica''. Caterina 
Stoppato also acknowledges support by FSE, Regione Lombardia and FIRB 
2012 ``Differential Geometry and Geometric Function Theory''.

\medskip
\normalsize

\bibliography{}

\enddocument